\documentclass[12pt, a4paper,leqno]{amsart}

\usepackage{amsmath,amsfonts,amsthm,amssymb,amsxtra}

\usepackage[numbers]{natbib}
\usepackage[colorlinks,citecolor=red,pagebackref,hypertexnames=false]{hyperref} 

\usepackage{tikz} 
\usepackage{xcolor}

\usepackage{marginnote}
\usepackage[makeroom]{cancel}
\newtheorem{theorem}{Theorem}
\numberwithin{theorem}{section}

\newtheorem{proposition}[theorem]{Proposition}

\newtheorem{lemma}[theorem]{Lemma}

\newtheorem{corollary}[theorem]{Corollary}

\theoremstyle{definition}

\newtheorem{definition}{Definition}[section]

\newtheorem{assumption}[]{Assumption}

\theoremstyle{remark}
\newtheorem{remark}[theorem]{Remark}

\newcommand{\C}{\mathbb{C}}

\renewcommand{\epsilon}{\varepsilon}

\newcommand{\R}{\mathbb{R}}

\newcommand{\Z}{\mathbb{Z}}
\newcommand{\om}{\omega}
\newcommand{\Om}{\Omega}
\newcommand{\lat}{\mathcal{L}}

\newcommand{\lan}{\langle}
\newcommand{\ran}{\rangle}

\newcommand{\pname}{\text{Landscape admissible}}
\newcommand{\G}{\Gamma}
\newcommand{\tildeW}{\tilde W}

\newcommand{\vertiii}[1]{{\left\vert\kern-0.25ex\left\vert\kern-0.25ex\left\vert #1 
    \right\vert\kern-0.25ex\right\vert\kern-0.25ex\right\vert}}

\DeclareMathOperator{\Tr}{Tr}

\DeclareMathOperator{\den}{den}

\newcommand{\ipc}[2]{\left \langle #1 , \ #2 \right \rangle }
\newcommand{\eps}{\varepsilon}

\newcommand{\rmd}{d}

\newcommand{\wt}{\widetilde}

\makeatletter
\def\blfootnote{\gdef\@thefnmark{}\@footnotetext}
\makeatother
\begin{document}

\title{On a Novel Effective Equation of the Reduced Hartree-Fock Theory}
\author{Ilias Chenn, Svitlana Mayboroda, Wei Wang, Shiwen Zhang}

\maketitle
\date{ }

\newcommand{\Addresses}{{
  \bigskip
  \vskip 0.08in \noindent --------------------------------------

  \footnotesize

\medskip

I.~Chenn, \textsc{Department of Mathematics, Massachusetts Institute of Technology, 2-252b, 77 Massachusetts Avenue, 
Cambridge, MA  4307 USA }\par\nopagebreak
  \textit{E-mail address}: \texttt{nehcili@mit.edu}
  
\vskip 0.4cm

  S.~Mayboroda, \textsc{School of Mathematics, University of Minnesota, 206 Church St SE, Minneapolis, MN 55455 USA}\par\nopagebreak
  \textit{E-mail address}: \texttt{svitlana@math.umn.edu}

\vskip 0.4cm

  W. ~Wang, \textsc{School of Mathematics, University of Minnesota, 206 Church St SE, Minneapolis, MN 55455 USA}\par\nopagebreak
  \textit{E-mail address}: \texttt{wang9585@umn.edu }
  
\vskip 0.4cm

S.~Zhang, \textsc{School of Mathematics, University of Minnesota, 206 Church St SE, Minneapolis, MN 55455 USA}\par\nopagebreak
  \textit{E-mail address}: \texttt{zhan7294@umn.edu }
  }}

\begin{abstract}
We show that there is an one-to-one correspondence between solutions to the Poisson-Landscape equations and the reduced Hartree-Fock equations in the semi-classical limit at low temperature. Moreover, we prove that the difference between the two corresponding solutions is small by providing explicit estimates.  
\end{abstract}

 \tableofcontents

\section{Introduction}
\subsection{Reduced Hartree-Fock equation}
Despite the success of density functional theory (DFT), its computational difficulties remain a major bottleneck. Filoche and Mayboroda initiated a series of recent works on the Landscape function \cite{FM} which led to a further simplification of the density functional theory by introducing the Poisson-Landscape (PL) equation \cite{PS1, PS2}. The Landscape theory and numerical simulations \cite{ADFJM1, ADFJM2, ADFJM3, PS1, PS2, ADFJM2, WZ} suggest that solving the PL equation can be an efficient and accurate replacement of the original DFT. This success undoubtedly demands a rigorous mathematical justification and a theoretical foundation.

DFT originated as a systematic way to study large many-body quantum system by using a self-consistent 1-body approximation. Parallel to its development, a number of effective theories existed along with DFT; examples include the Hartree-Fock theory, the Bardeen–Cooper–Schrieffer (BCS) theory, and the Thomas-Fermi theory of electrons. While DFT enjoyed a similar energy functional as the more complex Hartree-Fock theory and the BCS theory, inheriting a form of accuracy, it also gravitated towards the Thomas-Fermi theory to study the simpler electron density instead of density matrices. Owing to these characteristics, the Kohn-Sham (KS) energy and equation of DFT was developed \cite{HK, KS}. These equations and their related theory have become a mainstay of modern condensed matter physics. Some notable areas of application include semi-conductor design, deformation theory in solid mechanics, and quantum chemistry. In the mean time, a plethora of mathematical studies also ensued, for example, see \cite{ELu, HS, Levy, Levy2, Lieb3, Lieb,LS1,LS2,MJ,So2}. 

The KS equation is a set of functional equations for the electron density $\rho$, which is often simplified to the reduced Hartree-Fock equation (REHF) to illuminate core mathematical properties while maintaining its key features (see, for example, \cite{CDL,CLL,CSL,CS,FNV,LS,So1}). This is achieved by ignoring the exchange-correlations terms in the KS equation. In the same spirit,  we will also consider this simplified REHF in our work and be consistent with the aforementioned Landscape theory in \cite{PS1, PS2}. 

Consider a semi-conductor at a positive temperature, $\beta^{-1}$, with dopant density $\kappa$, and a band-offset potential $V$. We choose physical units such that as many physical constants are set to $1$ as possible. In this case, the REHF equation states that the material's electron density, $\rho$, is given by
\begin{equation}
    \rho = \den f_{\rm FD}(\beta(-\Delta + V - \phi -\mu)), \label{eqn:REHF-rho-original}
\end{equation}
where $\mu$ is the chemical potential/Fermi energy, $f_{\rm FD}$ is the Fermi-Dirac distribution
\begin{equation}
    f_{\rm FD}(\lambda) = \frac{1}{1+e^\lambda},
\end{equation}
$\phi$ is the electric potential solving the Poisson equation
\begin{equation}
    -\Delta \phi = \kappa - \rho, \label{eqn:REHF-phi}
\end{equation}
and $\den$ is the density operator defined via 
\begin{equation}
    (\den A)(x) = A(x,x), \label{eqn:den-def}
\end{equation} where $A$ is an operator on $L^2(\R^3)$ and $A(x,y)$ is the integral kernel of $A$ (see Appendix \ref{sec:per-vol-set-up} for more details).  If $A$ has a full set of eigenbasis $\phi_i$ with eigenvalues $\lambda_i$, then $\den A$ has the more familiar expression
\begin{equation}
    (\den A)(x) = \sum_i \lambda_i |\phi_i|^2(x).
\end{equation} 

We remark that while equation \eqref{eqn:REHF-rho-original} is an equation for microscopic electronic structures of matter, dopant potentials and band-offsets often vary on another larger mesoscopic scale. A precise formulation of the problem would require a homogenized version of \eqref{eqn:REHF-rho-original} where mesoscopic parameters such as the dielectric operator emerge. However, we will make the possibly unphysical assumption that \eqref{eqn:REHF-rho-original} is already homogenized and the dielectric constant is $1$ purely for mathematical simplicity (see remark \ref{rmk:dielectric-constant-dont-matter}).

Moreover, we further restrict ourselves to the semi-classical regime and modify \eqref{eqn:REHF-rho-original} as
\begin{equation}
    \rho = \den f_{\rm FD}(\beta(-\eps^2 \Delta + V - \phi -\mu)), \label{eqn:REHF-rho}
\end{equation}
where $\eps \ll 1$ is the semi-classical parameter. In this regime, one natural effective equation for \eqref{eqn:REHF-rho} is 
\begin{equation}
    \rho = \frac{1}{(2\pi\eps)^{3}} \int_{\R^3} dp f_{\rm FD}(\beta(p^2 + V - \phi  - \mu)), \label{eqn:sc-old}
\end{equation}
where $\phi$ solves \eqref{eqn:REHF-phi} as before.
However, in semi-conductor models, the band-offset potential $V$ is piecewise constant. This renders semi-classical analysis potentially ineffective. That is, the na\"{i}ve  error, $O(\eps^{-3+1}\nabla V)$, of the difference between the right hand side of \eqref{eqn:REHF-rho} and the right hand side of \eqref{eqn:sc-old} cannot be meaningfully controlled. Consequently, a form of regularization is needed. The Poisson-Landscape equation presents a regularization method that preserves both the spectrum of the Hamiltonian and the density $\rho$ (more details can be found in the proof of Theorem \ref{thm:eff-eqn}).

\subsection{Landscape theory and the Poisson-Landscape equation}
In one view, the Landscape theory presents a partial diagonalization of the Schr\"{o}dinger Hamiltonian $H = -\epsilon^2\Delta + V - \phi$ \cite{ADFJM3}. In \cite{FM}, if $H > 0$, the Landscape function $u$ is defined as
\begin{equation}
    Hu = 1,
\end{equation}
and the Landscape potential $W$ is defined as
\begin{equation}
    W = 1/u. \label{eqn:landscape-pot-W-1u-def-in-intro}
\end{equation}
Conjugating $H$ by $u$, we obtain
\begin{equation}
    u^{-1} H u = -\eps^2\Delta - 2u^{-1}\eps\nabla u \cdot \eps\nabla + W. \label{eqn:u-inv-H-u}
\end{equation}
We remark here that $u^{-1} H u$ has the same spectrum as $H$. This forms the basis for isospectral regularization as mentioned at the end of the previous section. Ignoring the drift term in $u^{-1}Hu$, this suggest that we should modify equation \eqref{eqn:sc-old} as
\begin{equation}
    \rho =  \frac{1}{(2\pi\eps)^{3}} \int_{\R^3} dp f_{\rm FD}(\beta(p^2 + W  - \mu)), \label{eqn:PL-rho-form}
\end{equation}
where
\begin{align}
    &W = 1/u, \\
    &(-\eps^2 \Delta + V - \phi)u = 1.
\end{align}
This equation was proposed and studied in \cite{PS1, PS2}. Together with \eqref{eqn:REHF-phi}, they bear the name Poisson-Landscape (PL) equation. 

The PL equation was proposed as a computational simplification rather than a regularization method for the semi-classical expansion initially. Numerical solution to the REHF equation requires an extensive computation of a large number of eigenvalues and eigenfunctions of the Hamiltonian $H$. Although various eigensolvers have been developed for this purpose (for a survey, see \cite{BDDRV, Saad}), such a direct computation remains a challenge in large-scale systems, particularly in high dimensions. 
In the specific setting of semi-conductor physics with random potentials, the Landscape function $u$ alleviates this problem through the approximation that the $i$-th lowest eigenvalue $E_i$ of the Hamiltonian $H$ can be numerically predicted by the $i$-th smallest local minimum of the landscape potential $W$ (defined in \eqref{eqn:landscape-pot-W-1u-def-in-intro}), $W_i$:
\begin{equation}
    E_i \approx \left( 1+ \frac{d}{4} \right) W_i,
\end{equation}
where $d$ is the spatial dimension (see \cite{ADFJM1}). Following this success, \cite{ADFJM1} showed further that the number of eigenvalues below $E$, $N_V(E)$, of $H$ can be approximated by
\begin{equation}
    N_V(E) \approx   \frac{1}{(2\pi\eps)^{3}} \int_{\R^3 \times \Om} dp dx\, 1_{\{p^2 + W(x) \leq E\}} \label{eqn:eff-Weyl}
\end{equation}
numerically. This approximation enjoys a more accurate prediction than the usual Weyl's law on average. We note that the left hand side of \eqref{eqn:eff-Weyl} is
\begin{equation}
    \int_{\R^3}  \rho_{T=0,\mu=E}, 
\end{equation}
where $\rho_{T=0,\mu=E}$ is the electron density at zero temperature with $\mu = E$ (cf. \eqref{eqn:REHF-rho}). Consequently, we expect that the solutions to the Poisson-Landscape equation \eqref{eqn:PL-rho-form} are good approximations to the density of electrons. 

Up to now, many of the stated advantages of the Landscape theory have mostly been proven useful for numerical purposes. Hence, the goal of the current work is to introduce a rigorous treatment of the Poisson-Landscape equation as an effective equation of the REHF equation in the semi-classical limit. Other related rigorous mathematical treatments of the Landscape theory can be found in \cite{ADFJM2, DFM, WZ}.

\subsection{Results}
We limit ourselves to the periodic setting in which physical quantities are periodic on $\Om =\R^3/(L\Z)^3\cong [0,L]^3$, while the quantum states are on $\R^3$. That is, quantities such as $\rho, \kappa, V$ or $\phi$ are periodic while the associated operators, such as $H = -\eps^2\Delta+V-\phi$, act on $L^2(\R^3)$ (see Appendix \ref{sec:per-vol-set-up} for more details). 

Moreover, let $X = \R^3$ or $\Om$
and $L^p(X;{\mathbb F})$ be the usual $L^p$ space of ${\mathbb F}$ valued functions on $X$, where ${\mathbb F}=\R$ or $\C$. In the special case when $\mathbb{F}=\C$, we denote $L^p(X)=L^p(X;\C)$. We endow $L^p(X; \mathbb{F})$ with its standard $p$-norms. Similarly, we equip $L^2(X; \mathbb{F})$ its standard inner product. Due to the periodic nature of $\Om$, we identify $L^2(\Om; \mathbb{F})$ with
\begin{equation}
    \left\{ f \in L^2_{\rm loc}(\R^3;  \mathbb{F}) : \text{$f$ is $(L\Z)^3$ periodic and } \int_\Om |f|^2 < \infty \right\}.
\end{equation} 
We let $H^s(\Om;{\mathbb F}) 
\subset L^2(\Om;{\mathbb F})$ denote the the associated Sobolev spaces of order $s$ with periodic boundary conditions. The identification of $H^s(\Om; \mathbb{F})$ with $H^s$ periodic functions on $\R^3$ persists. When $\mathbb{F} = \C$, we will suppress the symbol $\C$.  
The conversion from $L^2(\R^3; \mathbb{F})$ to $L^2(\Om; \mathbb{F})$ is done via the density operator $\den$, introduced in \eqref{eqn:den-def}. That is, the $\den$ of a periodic operator on $L^2(\R^3)$ is a periodic function, with fundamental domain $\Om$. Next, we restrict our study to the following type of piecewise constant potentials, which can be viewed as a (hence any) realization of a random potential of Anderson type.  

\begin{definition} \label{eqn:pname-def}
Let $0 < L \in \Z$. An $(L\Z)^3$ periodic potential $V$ is called {\pname} if $V$ is a strictly positive and piecewise constant, given by 
\begin{equation}\label{eq:pwV}
    V(x)=\sum_{j\in \Z^3}\omega_j\, \chi(x-j), \ \ {\rm for}\  \ x\in \Omega  ,
\end{equation}
where $0 < \omega_j \in \R$ is $(L\Z)^3$ periodic in $j$  and $\chi(x)$ is the indicator function of $[0,1)^3$. We note that a ${\pname}$ function $V$ is real valued by this definition.
\end{definition}

Through out the paper, we will write $A \lesssim B$ or $A = O(B)$ if $A \leq CB$ for some constant $C$ independent of $\eps$, $\delta$, $\beta$, $V$, and $\kappa$. In particular, we will write $A \approx B$ if $A \lesssim B \lesssim A$. We will also write $A \ll B$ or $A = o(B)$ if  
\begin{equation}
    \lim_{\delta \rightarrow 0} \lim_{\substack{\epsilon \rightarrow 0\\ \beta \rightarrow \infty}} \frac{A}{B} = 0. \label{eqn:ll-def}
\end{equation}
Whenever $\epsilon$ and $\beta$ are related (under Assumption \ref{ass:temperature} below), the inner limit in \eqref{eqn:ll-def} is taken under the said relationship.

Our first result shows that the density on the right hand side of \eqref{eqn:REHF-rho} can be approximated by the right hand side of \eqref{eqn:PL-rho-form}. This result will be proved in Section \ref{sec:expansion}. We begin by stating the necessary assumptions.

\begin{assumption}[Semi-classical regime] \label{ass:semi-classical}
The semi-classical parameter 
\begin{equation}
    \eps \ll 1.
\end{equation}
\end{assumption}

\begin{assumption} \label{ass:pname}
The external potential $V$ is periodic on $\Om$ and is {\pname} (in particular, it is real valued). Let $V_{\min}$ and $V_{\max}$ denote the minimum and maximum of $V$ on $\Om$, respectively. Let $\delta = V_{\max} - V_{\min}$. Then
\begin{equation}
    \eps \ll \delta \ll V_{\min} .\label{eqn:small-delta-ass}
\end{equation}
\end{assumption}

\begin{theorem} \label{thm:FD-leading-order}
Let $V$ be a {\pname} potential and assume that Assumptions \ref{ass:semi-classical} and \ref{ass:pname} hold. In addition, assume that 
\begin{enumerate}
    \item $\beta > 0$ is sufficiently large,
    \item  $\phi \in H^2(\Om;\R)$ and $\|\phi\|_{H^2(\Om)} \lesssim \delta$, 
    \item $V - \mu \geq C > 0$, where $C$ is a constant independent of $\delta$ and $\epsilon$.
\end{enumerate}
Then there exists $V_{\rm cut} \in \R$  
with
\begin{align}
  \delta^{1/4} \lesssim V_{\min} - V_{\rm cut} \lesssim \delta^{1/4}, \label{eqn:cut-size}
\end{align}
such that
\begin{align}
  \den f_{\rm FD} & (\beta (-\epsilon^2 \Delta +V - \phi - \mu)) \notag \\
  &=   \frac{1}{(2\pi\eps)^{3}}\int_{\R^3} dp f_{\rm FD}(\beta(p^2 + W_1 + V_{\rm cut}-\mu)) + R_1 \label{eqn:FD-semi-classical-leading-orders} \\
  &=    \frac{1}{(2\pi\eps)^{3}} \int_{\R^3} dp f_{\rm FD}(\beta(p^2 + W_2 - \phi + V_{\rm cut}-\mu)) + R_2, \label{eqn:FD-semi-classical-leading-orders-2}
\end{align}
where
\begin{equation}
    \|R_1\|_{L^2(\Om)},\|R_2\|_{L^2(\Om)} \lesssim \eps^{-3+1/2} \beta^{-1} e^{-\beta(V_{\rm cut}-\mu-\delta^{1/4})}
\end{equation}
and $W_1=1/u_1$ and $W_2 = 1/u_2$ solve
\begin{align}
    &(-\eps^2 \Delta+ (V-\phi-V_{\rm cut}))u_1=1, \label{eqn:shifted-Landscape-eqn} \\
    &(-\eps^2 \Delta+ (V-V_{\rm cut}))u_2 =1. \label{eqn:shifted-Landscape-eqn-2} 
\end{align}
\end{theorem}

Theorem \ref{thm:FD-leading-order} provides the foundation for a rigorous justification of the PL equation. In addition, \eqref{eqn:shifted-Landscape-eqn-2} suggests that a simpler effective equation is also possible. More precisely, let
\begin{align}
    &F_{\rm REHF}(\phi, \mu) := \den f_{\rm FD}(\beta(-\eps^2 \Delta + V - \phi -\mu)), \label{eqn:F-REHF-def}\\
    &F_{\rm PL}(\phi,\mu) := \frac{1}{(2\pi\eps)^{3}} \int_{\R^3} dp f_{\rm FD}(\beta(p^2 + W_1 + V_{\rm cut}-\mu)), \label{eqn:F-PL-def}\\
    & F_{\rm LSC}(\phi,\mu) := \frac{1}{(2\pi\eps)^{3}} \int_{\R^3} dp f_{\rm FD}(\beta(p^2 + W_2 - \phi + V_{\rm cut}-\mu)), \label{eqn:F-LSC-def}
\end{align}
where $W_1 =1/u_1$ and $W_2 = 1/u_2$ and $u_1$ and $u_2$ are given in \eqref{eqn:shifted-Landscape-eqn} and \eqref{eqn:shifted-Landscape-eqn-2}, respectively. LSC stands for ``Landscape regularized semi-classical'' and we will henceforth call this new $ F_{\rm LSC}$ the Landscape regularized semi-classical (LSC) regime. We note that  $ F_{\rm LSC}$  is a further simplification of $F_{\rm PL}$ and more closely resembles the semi-classical approximation \eqref{eqn:sc-old}.  Inserting $\rho = F(\phi,\mu)$ for $F = F_{\rm REHF}, F_{\rm PL},  F_{\rm LSC}$ into equation \eqref{eqn:REHF-phi}, we obtain the REHF, PL, LSC equation, respectively, for the electric potential $\phi$:
\begin{equation}
    -\Delta \phi = \kappa - F(\phi,\mu). \label{eqn:general-F-def}
\end{equation}

One advantage of $F=F_{\rm LCS}$ is that \eqref{eqn:general-F-def} is the Euler-Lagrange equation of a certain (energy) functional (see Appendix \ref{app:PL-existence}). This ensures that the linearization in $\phi$ is self-adjoint, whereas the linearization of $F_{\rm PL}$ is not self-adjoint in general. More importantly, the potential $W_2$ does not depend on $\phi$ and it only depends on the underlying material property due to $V$. One may further incorporate the doping features into $V$ the addition of an ansatz due to doping and electron density. That is, if $\rho_0$ is an a priori estimate for $\rho$, with associated electric potential $\phi_0$, we may look for solutions to \eqref{eqn:general-F-def} of the form $\rho = \rho_0 + \rho'$ and $\phi = \phi_0 + \phi'$. Substituting these expressions into \eqref{eqn:general-F-def} and upon minor modification, we obtain
\begin{equation}
    -\Delta \phi' =   \frac{1}{(2\pi\eps)^{3}} \int_{\R^3} dp f_{\rm FD}(\beta(p^2 + \widetilde W - \phi' + V_{\rm cut}-\mu)) - \rho_0
\end{equation}
where $\widetilde W = 1/\tilde u$ and $\tilde u$ solves
\begin{equation}
    (-\epsilon^2 \Delta + V - \phi_0 - V_{\rm cut}) \tilde u = 1.
\end{equation}
Hence, all the material and doping properties are stored in $\tilde W$, which is independent of $\phi'$.

Finally, to state our main result relating the REHF, PL, LSC equations and the associated electric fields, we specify additional assumptions.

\begin{assumption}[Low temperature] \label{ass:temperature}
There is some $K \in \R$ such that $0 < K < V_{\min}$ and the inverse temperature $\beta$ satisfies
\begin{equation}
    K < \frac{\log(\eps^{-3})}{\beta} < V_{\min}.
\end{equation}
\end{assumption}

\begin{assumption}[Conservation of charge] \label{ass:doping}
The doping potential $\kappa \in L^{2}(\Om;\R)$. Moreover, 
\begin{equation}
    \kappa_0 :=\frac{1}{|\Om|}\int_\Om \kappa \label{eqn:kappa-0-def}
\end{equation}
is a fixed constant.
\end{assumption}

\begin{theorem}[Main result] \label{thm:eff-eqn}
Let Assumptions \ref{ass:semi-classical} - \ref{ass:doping} hold. Assume that $(\phi_0, \mu)\in H^2(\Omega;\R)\times \R$ solves \eqref{eqn:general-F-def} with $F$ being any one of \eqref{eqn:F-REHF-def}, \eqref{eqn:F-PL-def}, or \eqref{eqn:F-LSC-def}, and
\begin{equation}
    \|\phi_0\|_{H^2(\Om)} \lesssim \delta. \label{eqn:ass:on-phi-0}
\end{equation}
Then there exists $C_1,C_2 > 0$ and a unique $\phi \in H^2(\Om;\R)$ such that $\|\phi_0 - \phi\|_{H^2(\Om)} \lesssim \epsilon^{C_1\delta^{1/4}}$ and $(\phi, \mu)$ solves \eqref{eqn:general-F-def} with $F$ being any other one of \eqref{eqn:F-REHF-def}, \eqref{eqn:F-PL-def}, or \eqref{eqn:F-LSC-def}. Moreover,
\begin{equation}
     \|\phi_0 - \phi\|_{H^2(\Om)} \lesssim \epsilon^{1/2-C_2\delta^{1/4}}. \label{eqn:PS-error}
\end{equation}
\end{theorem}

Theorem \ref{thm:eff-eqn} has an immediate corollary in terms of the density $\rho$. Rearranging \eqref{eqn:general-F-def}, the corresponding equations for the density are
\begin{align}
    \rho =& F(\phi, \mu), \label{eqn:general-F-rho-eqn1}\\
    -\Delta \phi =& \kappa - \rho. \label{eqn:general-F-rho-eqn2}
\end{align}

\begin{corollary} \label{cor:eff-eqn}
Let Assumptions \ref{ass:semi-classical} - \ref{ass:doping} hold. Assume that  $(\rho_0, \mu)\in (\kappa+H^{-2}(\Omega;\R))\times \R$ solves \eqref{eqn:general-F-rho-eqn1} - \eqref{eqn:general-F-rho-eqn2} with $F$ being any one of \eqref{eqn:F-REHF-def}, \eqref{eqn:F-PL-def}, or \eqref{eqn:F-LSC-def}, and
\begin{align}
    \|\kappa-\rho_0\|_{H^{-2}(\Om)} \lesssim \delta. \label{eqn:condition-on-existence}
\end{align}
Then there exists $C_1,C_2 > 0$ and a unique $\rho \in \kappa+H^{-2}(\Om;\R)$ such that $\|\rho_0 - \rho\|_{H^{-2}(\Om)} \lesssim \epsilon^{C_1\delta^{1/4}}$ and $(\rho, \mu)$ solves \eqref{eqn:general-F-rho-eqn1} - \eqref{eqn:general-F-rho-eqn2} with $F$ being any other one of \eqref{eqn:F-REHF-def}, \eqref{eqn:F-PL-def}, or \eqref{eqn:F-LSC-def}. Moreover,
\begin{align}
     \|\rho_0 - \rho\|_{H^{-2}(\Om)} \lesssim \epsilon^{1/2-C_2\delta^{1/4}}.
\end{align}
\end{corollary}

\begin{remark}
Corollary \ref{cor:eff-eqn} answers the challenge posed in the introduction. It justifies \cite{PS1, PS2} on a mathematically rigorous level in the semi-classical regime at low temperature (or large $\beta$).
\end{remark}

\begin{remark} \label{rmk:dielectric-constant-dont-matter} We noted in the paragraph before equation \eqref{eqn:REHF-rho} that the dielectric constant is taken to be $1$. However, as one will see from the proof of our main result Theorem \ref{thm:eff-eqn}, so long as the dielectric constant is strictly positive, the same conclusion can be derived, albeit with more cumbersome proofs.
\end{remark}

\begin{remark}
Note that in both Theorem \ref{thm:eff-eqn} and Corollary \ref{cor:eff-eqn}, a solution to \eqref{eqn:general-F-def} is a pair: either $(\phi, \mu)$ or $(\rho, \mu)$. Because of this particular view of solution, equation \eqref{eqn:general-F-def} has an important dilation symmetry (detailed below). Moreover, since $\kappa$ is real, another important complex conjugation symmetry exists. We now discuss these two symmetries and their consequences in light of Theorem \ref{thm:eff-eqn} and Corollary \ref{cor:eff-eqn}.
\begin{enumerate}
    \item (dilation symmetry)
    \begin{align}
        (\phi, \mu) \mapsto (\phi+t, \mu-t)
    \end{align}
    for $t \in \R$.
    \item (complex conjugation) If $\kappa$ and $V$ are real valued and $\mu \in \R$, then
    \begin{align}
        (\phi, \mu) \mapsto (\mathcal{C} \phi, \mu) \label{eqn:cc-sym}
    \end{align}
    is a symmetry of \eqref{eqn:general-F-def} where $\mathcal{C} \phi = \bar \phi$ is the complex conjugation of $\phi$.
\end{enumerate}
Dilation requires one to regard all solutions $(\phi, \mu)$ related by a dilation as a single solution. In this way, the uniqueness of solution is regarded as uniqueness among an equivalence class. Nevertheless, since we fixed $\mu$ in Theorem \ref{thm:eff-eqn} and Corollary \ref{cor:eff-eqn}, a particular representative of the equivalence class is chosen and there is no ambiguity in the word ``unique''. Perhaps a better way to view this is to consider $\phi+\mu$ as the solution instead of $(\phi, \mu)$. In this way, one avoids the equivalence class description. Nevertheless, since we are interested in the difference of two solutions (see \eqref{eqn:PS-error}), any choice of either point of view causes no harm. Moreover, the complex conjugation symmetry (and the uniqueness of solution) ensures that any solution to \eqref{eqn:general-F-def} with real $\kappa, V$ and $\mu$ is necessarily real. Thus, the conclusions regarding the reality of $\phi$ and $\rho$ in Theorem \ref{thm:eff-eqn} and Corollary \ref{cor:eff-eqn}, respectively, are in fact superfluous. 

One also note that the PL equation with \eqref{eqn:F-PL-def} does not have the dilation symmetry, contrasting the case of \eqref{eqn:F-REHF-def} and \eqref{eqn:F-LSC-def}. Whether this difference makes numerical approximations using \eqref{eqn:F-PL-def} less desirable is out of the scope of this paper, since \eqref{eqn:F-PL-def} respects the dilation symmetry in leading order $\eps$ if $\eps \ll 1$.
\end{remark}

Theorem \ref{thm:eff-eqn} could help us to prove existence of solutions for the three classes of equations REHF, PL, and LSC simultaneously. However, we were unable to prove the smallness assumption \eqref{eqn:ass:on-phi-0} in general. Though, we believe this condition should hold in many cases if $\|\kappa\|_{H^2(\Om)} \lesssim \delta$ (for related results, see \cite{PN, Lev, CZ}). Nevertheless, we provide an existence result to the simplest case, the LSC equations, via variational principle for completeness sake. Since this type of existence result is well studied in the literature, we will not enumerate all previous works. The interested reader is referred to, for example, \cite{AACan, CLeBL, CLeBL2, CS1, Nier}.

\begin{theorem}[LSC existence]
 If $\kappa - \kappa_0 \in H^{-2}(\Om;\R)$ (see \eqref{eqn:kappa-0-def} for definition of $\kappa_0$), there exists a solution $(\phi, \mu) \in H^2(\Om;\R) \times \R$ to the LSC equation \eqref{eqn:F-LSC-def}.
\end{theorem}
\begin{proof}
This is a direct corollary of Theorem \ref{thm:PL-existence-uniqueness}.
\end{proof}

Theorem \ref{thm:FD-leading-order} is proved in Section \ref{sec:expansion} following an analysis of the Landscape potential in Section \ref{sec:Landscape-anal}.  In order to use Theorem \ref{thm:FD-leading-order} to prove Theorem \ref{thm:eff-eqn}, we digress briefly in Section \ref{sec:integrability} to establish a relationship between the parameters $\epsilon, \beta, \mu$, etc. as a result of the constraint of the integrability condition
\begin{equation}
    \int_\Om \kappa = \int_\Om F(\phi, \mu), \label{eqn:integrability-cond}
\end{equation}
obtained by integrating \eqref{eqn:general-F-def} over $\Om$. The results from Section \ref{sec:integrability} and the assumptions of Theorem \ref{thm:eff-eqn} provide the proper scaling regime to control our estimates. Finally, The core proof of Theorem \ref{thm:eff-eqn} is given in Section \ref{sec:main-result-proof} while its related technical details are collected in subsequent sections: its linear analysis is given in Sections \ref{sec:lin-anal} and its nonlinear analysis is provided in Sections \ref{sec:nonlin-anal}.

\textbf{Acknowledgments.} The authors are grateful for stimulating discussions with D. N. Arnold, J.-P. Banon, M. Filoche, D. Jerison, A. Julia. The first author thanks I. M. Sigal for many helpful insights and guidance.

Chenn is supported in part by a Simons Foundation Grant 601948 DJ and a PDF fellowship from NSERC/Cette recherche a \'{e}t\'{e} financ\'{e}e par le CRSNG. 
Mayboroda is supported by NSF DMS 1839077 and the Simons Collaborations in MPS 563916, SM.
Wang is supported by Simons Foundation grant 601937, DNA. Zhang
is supported in part by the NSF grants DMS1344235, DMS-1839077, and Simons Foundation grant 563916, SM.

\section{Landscape function in the semi-classical regime} \label{sec:Landscape-anal}
In this section, we will obtain several estimates for the Landscape function $u$ in the semi-classical regime. These estimates will play an important role in the proof of the main result. The Landscape function $u$ is the solution to
\begin{equation}
    (-\eps^2 \Delta+V-\phi)u  = 1 \label{eqn:Landscape-original}
\end{equation}
on $\Om = [0,L]^3$  with periodic boundary condition,  where $V$ is {\pname} and $\phi \in H^2(\Om)$. Let us recall that $\delta = V_{\max} - V_{\min}$ where $V_{\rm max}$ and $V_{\rm min}$ are the max and min of $V$ on $\Om$, respectively. Then we have the following main result.

\begin{theorem} \label{thm:useful-L-p-est-on-DW}
Let $2 \leq p \leq \infty$. Suppose that $V$ is {\pname} and $\phi \in H^2(\Om)$. Moreover, assume that $\epsilon \ll \delta \ll V_{\rm min} \leq 1$ and $\|\phi\|_{H^2(\Om)} \lesssim \delta$. Let $W = 1/u$ where $u$ solves the Landscape function \eqref{eqn:Landscape-original} with periodic boundary condition on $\Om$, then
\begin{align}
    & \|\nabla W\|_{L^p(\Om)} \leq C \delta  V_{\max}^{1/2-1/p} \epsilon^{-\frac{p-1}{p}} , \label{eqn:Lp-of-nabla-u} \\
    & \|\Delta W\|_{L^p(\Om)} \leq C \delta  V_{\max}^{1-4/p} \eps^{-\frac{2p-1}{p}}, \label{eqn:Lp-of-Delta-u} 
\end{align}
where $C$ depends on $d$ and $p$ only.
\end{theorem}

We start by estimating $\nabla^s u$ in the $L^2$ and $L^\infty$ norms first for $s=0,1,2$ below. Theorem \ref{thm:useful-L-p-est-on-DW} is proved at the end of this section by interpolation. As a remark, we will write the $L^p(\Om)$ and $H^s(\Om)$ norms as $\| \cdot \|_p$ and $\| \cdot \|_{H^s}$, respectively, when there is no ambiguity.

\begin{proposition} \label{pro:Hs-estimate-u}
Suppose that $V$ is {\pname} and $\phi \in H^2(\Om)$. Moreover, assume that $\epsilon \ll \delta \ll V_{\rm min} \leq 1$ and $\|\phi\|_{H^2(\Om)} \lesssim \delta$. 
If $u$ solves \eqref{eqn:Landscape-original}  with periodic boundary condition on $\Om$, then there is a constant $C$ such that 
\begin{align}
    & \|u\|_{L^2(\Om)}\le C\frac{1}{V_{\min}},  \label{eqn:pro:Hs-estimate-u-0} \\ 
    & \|\nabla u\|_{L^2(\Om)}\le C\frac{\delta}{V_{\min}^{2}}\varepsilon^{-1/2},   \label{eqn:pro:Hs-estimate-u-1} \\ 
    & \|\Delta u\|_{L^2(\Om)}\le C\frac{\delta}{V_{\min}^{3/2}}\varepsilon^{-3/2}. \label{eqn:pro:Hs-estimate-u-2}
\end{align}
\end{proposition}

\begin{proof}[Proof of Proposition \ref{pro:Hs-estimate-u}]
The first inequality \eqref{eqn:pro:Hs-estimate-u-0} follows from the fact that the Hamiltonian $H=-\eps^2 \Delta + V-\phi$ is bounded below by  $V_{\min}-C\delta \gtrsim \frac{1}{2}V_{\min}$ for some constant $C$. We prove \eqref{eqn:pro:Hs-estimate-u-1} and \eqref{eqn:pro:Hs-estimate-u-2} below.

Notice that $V$ is only discontinuous on a subset $\Om_0=\{x=(x_1,\cdots,x_d)\in \Om:\  x_j\in \Z\} \subset \Om$ and piecewise constant elsewhere. Let $\Omega_\eps$ be the $\eps$ neighborhood of the discontinuities of $V$:
\begin{equation}
    \Omega_\epsilon =\left\{ x\in \Om:\, |x-y|\le \eps,\ y\in \Om_0 \right\}. \label{eqn:Om-eps-def}
\end{equation}
It is easy to check that $|\Omega_\eps|\lesssim \eps\cdot L^d$ (where recall that $\Om$ is diffeomorphic to $[0, L]^d$). 
Let $\eta_\eps$ be a standard smooth bump function supported on $B_{\eps/2}(0)$ such that
\begin{equation}
   0\le \eta_\eps(x)\le  \eta_\eps(0)= \eps^{-d}, \|\eta_\eps\|_{L^1(\R^3)}= 1, \ |\nabla \eta_\eps|\lesssim \eps^{-d-1}. 
\end{equation}

First, we prove \eqref{eqn:pro:Hs-estimate-u-1} for $u_\eps:=\eta_\eps*( 1/{\wt V})$ where $\wt V=V-\phi$. Then we use $u_\eps$ to approximate $u$ for our estimates. Let $V_0$ denote the average of $\wt  V$ on $\Om$ and $\delta = V_{\rm max} - V_{\rm min}$.  Since $\|\phi\|_{H^2} \lesssim \delta$ by assumption,
\begin{equation}
    \|\widetilde V - V_0\|_{\infty} \lesssim \delta. \label{eqn:Vtilde-V-0-close}
\end{equation}
 
We rewrite 
\begin{equation}
    u_\eps = \frac{1}{V_0} + u_\epsilon', \label{eqn:u-epsilon-expand-def}
\end{equation}
where 
\begin{equation}
    u_\eps' = \eta_\eps * \left( \frac{1}{ \wt  V} - \frac{1}{ V_0} \right). \label{eqn:u-epsilon-prime-def}
\end{equation}
For any $x\in \Omega$, using \eqref{eqn:Vtilde-V-0-close}, we see that 
\begin{align*}
    |\nabla u_\eps'(x)|
    \le & \int \left|\nabla\eta_\eps(x-y)\, \left( \frac{1}{ \wt  V} - \frac{1}{V_0} \right)(y) \right|\, \rmd y \\
  \le & \max|\nabla\eta_\eps|  \frac{1}{V_{\min}^2}\, \int_{B_{\eps/2}(x)} |(\wt V-V_0)(y)|\, \rmd y  \\
    \le & \max|\nabla\eta_\eps|   \frac{1}{V_{\min}^2}\,   \|\wt V-V_0\|_\infty   |{B_{\eps/2}(x)} | \\
    \lesssim & \delta V_{\min}^{-2}\eps^{d}\eps^{-d-1} \\
    =& \delta\,  V_{\min}^{-2}\, \eps^{-1}.
\end{align*}
Thus, on $\Om_\eps$, we have 
\begin{align}
 \int_{\Omega_\eps}   |\nabla u_\eps'|^2\, \rmd x\, \lesssim \frac{\delta^2}{V_{\min}^4} \int_{\Omega_\eps}   |\eps^{-1}|^2\, \rmd x  
 \lesssim& \,  \frac{\delta^2}{V_{\min}^4} \eps^{-2}|\Omega_\eps| \\
 \le& \,  C\frac{\delta^2}{V_{\min}^4} \eps^{-1},  \label{eqn:grad-u-Om} 
\end{align} for some constant $C$. On the other hand, on $\Omega_\eps^C$, 
\begin{equation*}
    |\nabla u'_\eps|=\left|\, \eta_\eps*\frac{\nabla (\wt V)}{\wt V^2}\, \right|=\left|\, \eta_\eps*\frac{\nabla \phi}{\wt V^2}\, \right|\lesssim \frac{1}{V_{\min}^2}\, \left|\eta_\eps* {\nabla \phi}  \right|. 
\end{equation*}
It follows that
\begin{align}
    \int_{\Om_\eps^C}  |\nabla u'_\eps|^2\, \rmd x\lesssim& \frac{1}{V_{\min}^4} \int_{\Om_\eps^C}  \left|\eta_\eps* {\nabla \phi}  \right|^2\, \rmd x \\ \lesssim& \frac{1}{V_{\min}^4} \,  \|\eta_\eps\|_{L^1(\R^3)}^2 \|\nabla \, \phi\|_2^2 \\
  \lesssim&   \frac{\delta^2}{V_{\min}^4}. \label{eqn:grad-u-Om-c}  
\end{align} 
Combining \eqref{eqn:grad-u-Om} and \eqref{eqn:grad-u-Om-c}, we see that
\begin{equation}
       \|\nabla u_\eps\|_{2}=\|\nabla u'_\eps\|_{2}\le C\, \delta \,  V_{\min}^{-2}\, \eps^{-1/2}.  \label{eqn:u-eps-est}  
\end{equation}

Next, we decompose 
\begin{equation}
     u=u_\eps+u', \label{eqn:u-decomp-u-eps-u-prime}   
\end{equation}
where $u'$ is defined by this expression. We will control the size of $u'$ using energy estimates. We note that
\begin{equation}
   \ipc{u'}{Hu'}\ge \eps^2\|\nabla u'\|_2^2+\|\sqrt{V-\phi}\, u'\|_2^2   \ge \eps^2\|\nabla u'\|_2^2+\frac{1}{2}V_{\min}\|u'\|_2^2    .\label{eq:u'1} 
\end{equation}
 This provides an energy lower bound.  On the other hand,  
 \begin{equation}
      Hu'=Hu-Hu_\eps=\eps^2\Delta u_\eps+1-\, \wt V\, u_\eps. \label{eqn:equation-for-u-prime}
 \end{equation}
It follows that
\begin{align}
    |\ipc{u'}{Hu'}|\le& \Big|\ipc{\nabla u'}{\eps^2\nabla u_\eps}\Big|+\Big|\ipc{u'}{1-\, \wt V\, u_\eps}\Big|\\
    \le & \eps^2\|\nabla u'\|_{2}\, \| \nabla u_\eps\|_{2}\, +\, \|u'\|_{2}\, \|{1-\, \wt V\, u_\eps}\|_2 . \label{eq:u'2}
\end{align}
 This is an energy upper bound. 
We now estimate the term $\|{1-\, \wt V\, u_\eps}\|_2$. We write $\wt  V = V_0 + \delta V'$ where $V_0$ is the mean of $V$ on $\Om$ and $V'$ is defined by this expression with
\begin{equation}
    \|V'\|_{\infty} \lesssim 1. \label{eqn:V-prime-O-1} 
\end{equation}
 Together with \eqref{eqn:u-epsilon-expand-def}, we note that
 \begin{equation}
      \wt  V u_\eps-1 =    (V_0 + \delta V')(1/V_0 + u_\eps')-1 = \delta \frac{V'}{V_0} + V_0 u_\eps' + \delta V' u_\eps'. \label{eqn:Vtilde-u-ep-1}
 \end{equation}
Using \eqref{eqn:Vtilde-V-0-close} and applying Young's inequality to \eqref{eqn:u-epsilon-prime-def}, we see that 
\begin{equation}
     \|u_\eps'\|_\infty\lesssim \frac{1}{V_{\min}^2}\|\wt V-V_0\|_{\infty}\lesssim \frac{\delta}{V_{\min}^2}. \label{eqn:u-eps-prime-sup} 
\end{equation}
Applying \eqref{eqn:V-prime-O-1} and \eqref{eqn:u-eps-prime-sup} to \eqref{eqn:Vtilde-u-ep-1}, we see that 
 \begin{equation}
      \|1- \wt  V u_\eps\|_\infty \lesssim {\delta}/{V_{\min}}. \label{eqn:1-minus-Vtilde-u-eps}  
 \end{equation}
Thus, on $\Omega_\eps$, 
\begin{equation}
    \int_{\Omega_\eps}|1-\, \wt V\, u_\eps|^2 \, \rmd x \lesssim \,  
  \frac{\delta^2}{V_{\min}^2}  |\Omega_\eps|
\lesssim\,   \frac{\delta^2}{V_{\min}^2}\, \epsilon. \label{eqn:1-Vu-Om}  
\end{equation}

For any $x \in \Omega_\eps^C$, 
\begin{align}
  \left| \frac{1}{\wt  V}(x)-u_\eps(x)\right|  \le & \,  \int_{\R^3} |\eta_\eps(y)|  \left|\frac{1}{\wt V}(x)-\frac{1}{\wt V}(x-y)\right|\, \rmd y \\
  \lesssim & \, \frac{1}{V_{\min}^2}\, \int_{\R^3} |\eta_\eps(y)|  \left|\phi(x)-\phi(x-y)\right|\, \rmd y . \label{eqn:1-by-V-minus-u-eps}
\end{align}
We remark that the domain of integration is in fact $B_{\epsilon/2}$ since $\eta_\eps$ is supported on a ball of radius $\epsilon/2$ at the origin. To estimate the last line, we have
\begin{align}
    |\phi(x) - \phi(x-y)| \leq& \int_0^1 |\nabla \phi(x-ty) \cdot y| dt \\
    \leq & |y| \int_0^1 |\nabla \phi(x-ty)| dt.  \label{eqn:phix-phix-y}
\end{align}
 Since $\eta_\eps(y)$ has support of radius $O(\epsilon)$ centered at the origin, it follows from equations \eqref{eqn:1-by-V-minus-u-eps} and \eqref{eqn:phix-phix-y} that 
\begin{align}
  \int_{\Omega^C_\eps}|1-\, \wt V\, u_\eps|^2 \, \rmd x \lesssim & \, V^2_{\max}\int_{\Omega^C_\eps} \left|\frac{1}{\wt  V}(x)-u_\eps(x) \right|^2 \, \rmd x \\
  \lesssim & \, \frac{\epsilon^2 V^2_{\max}}{V_{\min}^2} \int_{\Om_\eps^C}  \, \rmd x \left( \int_{\R^3}  \rmd y \int_0^1 dt |\eta_\eps(y)||\nabla \phi(x-ty)| \right)^2 \label{eqn:1-Vu-1}. 
\end{align}
We perform H\"{o}lder's inequality (in the $dtdy$-integral) on the integrand $|\eta_\eps(y)||\nabla \phi(x-ty)|$ via the grouping
\begin{equation}
    |\eta_\eps(y)||\nabla \phi(x-ty)| = |\eta_\eps(y)|^{1/2}   (|\eta_\eps(y)|^{1/2} |\nabla \phi(x-ty)|)  
\end{equation}
to obtain
\begin{align}
  &  \hspace{-2cm}\left( \int_{\R^3}  \rmd y  \int_0^1 dt |\eta_\eps(y)||\nabla \phi(x-ty)| \right)^2 \notag \\
    &\leq  \left( \int_{\R^3}  \rmd y \int_0^1 dt |\eta_\eps(y)| \right) \left(\int_{\R^3}  \rmd y \int_0^1 dt |\eta_\eps(y)||\nabla \phi(x-ty)|^2 \right) \\
    &=  \int_{\R^3}  \rmd y \int_0^1 dt |\eta_\eps(y)||\nabla \phi(x-ty)|^2.
\end{align}
Inserting this into equation \eqref{eqn:1-Vu-1}, we obtain
\begin{align}
  \int_{\Omega^C_\eps}|1-\, \wt V\, u_\eps|^2 \, \rmd x  \lesssim & \frac{\epsilon^2 V^2_{\max}}{V_{\min}^2} \int_{\Om_\eps^C}  \, \rmd x \int_{\R^3}  \rmd y \int_0^1 dt |\eta_\eps(y)||\nabla \phi(x-ty)|^2 \\
  \leq & C\, \epsilon^2 \|\nabla \phi\|_{2}^2\\
  \le& C\, \epsilon^2 \delta^2. \label{eqn:1-Vu-Om-C-1}  
\end{align}
Combining the estimates on $\Om_\eps$ \eqref{eqn:1-Vu-Om} and $\Om_\eps^C$ \eqref{eqn:1-Vu-Om-C-1}, we have 
\begin{equation}
       \|1-\wt V\, u_\eps\|_2 \le C\,V_{\min}^{-1} \delta\, \eps^{1/2}.  \label{eqn:1-Vu-Om-C}
\end{equation}
Together with \eqref{eq:u'1} and \eqref{eq:u'2}, and using $2ab \leq a^2+b^2$ for any real numbers $a,b$, we see that
\begin{align}
 \eps^2\|\nabla u'\|_2^2&+\frac{1}{2}V_{\min}\|u'\|_2^2  \label{eqn:energy-left} \\ 
 \le &\,  (\eps\|\nabla u'\|_{2})   (\eps\, C\,\delta V_{\min}^{-2}\eps^{-1/2})+(C\,\delta\, \eps^{1/2}V_{\min}^{-1}\, V_{\min}^{-1/2})   (V_{\min}^{1/2} \|u'\|_{2}/2)\\
 \le & \,  \frac{3}{4}\eps^2\|\nabla u'\|_{2}^2+\frac{V_{\min}}{4}\|u'\|_{2}^2+C' \delta^2\, V_{\min}^{-4}\eps. \label{eqn:energy-right}
\end{align}
 Subtracting the first two terms of \eqref{eqn:energy-right} from both \eqref{eqn:energy-right} and \eqref{eqn:energy-left},  we see that 
 \begin{equation*}
       \eps^2\|\nabla u'\|_{2}^2+V_{\min}\|u'\|_2^2\le C\, \delta^2  \, V_{\min}^{-4}\eps.   
 \end{equation*}
Therefore, 
\begin{equation}
      \|\nabla u'\|_{2}\lesssim  \delta\, V_{\min}^{-2}\eps^{-1/2} \text{ and } \|u'\|_2 \lesssim \delta  V_{\min}^{-5/2} \eps^{1/2}. \label{eqn:u-prime-est}  
\end{equation}
Combining with \eqref{eqn:u-eps-est} and \eqref{eqn:u-decomp-u-eps-u-prime}, we see that
\begin{equation*}
   \|\nabla u\|_2 \leq C\, \delta \,  V_{\min}^{-2}\eps^{-1/2}  
\end{equation*}
for some constant $C$. This proves \eqref{eqn:pro:Hs-estimate-u-1}.

 Finally, we estimate the $L^2$ norm of $-\Delta u$. Recall that $\wt  V = V-\phi$. Using equations \eqref{eqn:Landscape-original}, \eqref{eqn:1-Vu-Om-C}, and \eqref{eqn:u-prime-est}, we see that 
\begin{align}
    \eps^2\|\Delta u\|_2=\|1-\wt V\, u\|_2\le& \|1-\wt V\, u_\eps\|_2+\|\wt V\, u'\|_2\\
    \le& C\, V_{\min}^{-1 }\,  \delta\, \eps^{1/2} +V_{\max}\, C\, \delta\,  V_{\min}^{-5/2 }\eps^{1/2}.
    \label{eqn:Delta-u-eqn-for-est}
\end{align}
Therefore, 
\begin{equation*}
   \eps^2\|\Delta u\|_2 \leq C\, \delta\, V_{\min}^{-3/2} \eps^{1/2} .    
\end{equation*}
Dividing both sides by $\eps^2$ proves \eqref{eqn:pro:Hs-estimate-u-2}.
\end{proof}

\begin{proposition} \label{prop:sup-norm-grad-u}
Suppose that $V$ is {\pname} and $\phi \in H^2(\Om)$. Moreover, assume that $\epsilon \ll \delta \ll V_{\rm min} \leq 1$ and $\|\phi\|_{H^2(\Om)} \lesssim \delta$. 
If $u$ solves \eqref{eqn:Landscape-original} with periodic boundary condition,  then
\begin{align}
    & \|\nabla u\|_{L^\infty(\Om)} \leq C \frac{\delta }{V_{\min}^{3/2}}  \epsilon^{-1}, \label{eqn:Lsup-nabla-u} \\
    & \|\Delta u\|_{L^\infty(\Om)} \leq C\frac{\delta}{V_{\min}} \epsilon^{-2}, \label{eqn:Delta-om-sup}
\end{align}
where $C$ depends on $d$ and $L$ only.
\end{proposition}
\begin{proof}[Proof of Proposition \ref{prop:sup-norm-grad-u}]
Define $\om$ via $u(x) = \om(\epsilon^{-1} x)$. Since $u$ solves the Landscape equation \eqref{eqn:Landscape-original}, $\om$ solves
\begin{equation*}
    (-\Delta + V_\epsilon) \om = 1, 
\end{equation*}
where $V_\eps(x) = V(\epsilon x) - \phi(\eps x)$. Moreover,
\begin{equation}
      \|\nabla^s u\|_{L^\infty(\Om)}  = \epsilon^{-s} \|\nabla^s \om \|_{L^\infty(\epsilon^{-1} \Om)}. \label{eqn:transfer-of-eps}  
\end{equation}
Consequently, we estimate the sup-norm of $\nabla \om$. As before, let $V_0$ denote the average of $V$ on $\Om$. We decompose
\begin{equation*}
     V_\epsilon = V_0 + \delta V_\eps', 
\end{equation*}
where $V_\eps'$ is defined by this expression. We remark that the mean of $V_\eps'$ over $\eps^{-1} \Om$ is $0$ and 
\begin{equation}
        -1 \leq V_\eps' \leq 1. \label{eqn:V-eps-prime-sup-bound}
\end{equation}
Let $H_0 = -\Delta + V_0$ and $R = H_0^{-1}$. We see that
\begin{align}
    \om =& (H_0 + \delta V_\eps' )^{-1} 1 \\
      =& \sum_{n \geq 0} (-1)^n \delta^n (RV_\eps')^n R 1 \\
      =& \frac{1}{V_0} \sum_{n \geq 0} (-1)^n \delta^n (RV_\eps')^n 1.  \label{eqn:omg-series-expansion}
\end{align}
It follows that
\begin{equation}
    \nabla \om =  \frac{1}{V_0} \sum_{n \geq 1} (-1)^n \delta^n \nabla (RV_\eps')^n 1 . \label{eqn:nabla-u-expansion}  
\end{equation}
We claim that
\begin{align}
    & \|R\|_ {L^\infty(\eps^{-1}\Om) \rightarrow L^\infty(\eps^{-1}\Om)} \leq \frac{C_1}{V_0} \label{eqn:R-Linfty},\\
    & \| \nabla R V_\eps' \|_{L^\infty} \leq C_2\frac{\delta}{\sqrt{V_0}}, \label{eqn:leading-resolvent-in-Linfty}
\end{align}
for some constants $C_1$ and $C_2$. For the sake of continuity, we defer the proof of the claims to the end of this section as they are simple corollaries of Young's inequality. Since the integral kernel of $R$ is positive and using equations \eqref{eqn:V-eps-prime-sup-bound}, \eqref{eqn:R-Linfty}, and \eqref{eqn:leading-resolvent-in-Linfty}, we see that \eqref{eqn:nabla-u-expansion} can be estimated as
\begin{equation*}
       \|\nabla \om \|_{\infty} \leq \frac{C}{V_0(1 -\delta/V_0)}   \frac{\delta}{\sqrt{V_0}} \leq C \frac{\delta}{V_{\min}^{3/2}},
\end{equation*}
since $\delta/V_0 \leq \delta/V_{\min} \ll 1$. Together with equation \eqref{eqn:transfer-of-eps}, equation \eqref{eqn:Lsup-nabla-u} is proved pending claims \eqref{eqn:R-Linfty} and \eqref{eqn:leading-resolvent-in-Linfty}.

Now we prove the claims \eqref{eqn:R-Linfty} and \eqref{eqn:leading-resolvent-in-Linfty}. Equation \eqref{eqn:R-Linfty} is standard. For the sake of completeness, we carry out the corresponding estimates. Let $g(x) = \frac{e^{-\sqrt{V_0}|x|}}{ 4\pi|x|}$. Then
\begin{equation*}
    (Rf)(x) = g * f.
\end{equation*}
It follows from Young's inequality  that
\begin{equation}
    \|R\|_ {L^\infty(\eps^{-1}\Om) \rightarrow L^\infty(\eps^{-1}\Om)} \leq C\int_{\R^3} dx \frac{e^{-\sqrt{V_0}|x|}}{4\pi|x|} = \frac{C}{V_0}  
\end{equation}
for some constant $C$. Similarly, we estimate \eqref{eqn:leading-resolvent-in-Linfty}. We note that
\begin{equation}
      (\nabla R V_\eps')(x) = (\nabla g) * V_\eps'.  
\end{equation}
Since 
\begin{equation}
      \int_{\R^3} |\nabla g| =  \int_{\R^3} dx g(x)(\sqrt{V_0}+|x|^{-1}) = \frac{C}{\sqrt{V_0}}
\end{equation}
for some universal constant $C$, claim \eqref{eqn:leading-resolvent-in-Linfty} follows by Young's inequality.

Finally, using 
\[ -\Delta \om  = 1-V_\eps \om,\]
we see that 
\begin{equation}
      \|\Delta \om\|_\infty = \|1-V_\eps \om\|_\infty. \label{eqn:Delta-om-1}  
\end{equation}
By equation \eqref{eqn:omg-series-expansion}, \eqref{eqn:R-Linfty}, and the fact $\delta = V_{\max} - V_{\min} \ll V_{\min}$,
\begin{equation}
     \|1-V_\eps \om\|_\infty \leq C\sum_{n \geq 1} \delta^n \|R\|_\infty^{n} \leq C_1 \frac{\delta/V_{\min}}{1-C_2\delta/V_{\min}} \leq C \frac{\delta}{V_{\min}},
\end{equation}
for some constant $C$. Together with \eqref{eqn:transfer-of-eps} and \eqref{eqn:Delta-om-1}, this proves \eqref{eqn:Delta-om-sup}. The proof of Proposition \ref{prop:sup-norm-grad-u} is now complete.
\end{proof}

\begin{proof}[Proof of Theorem \ref{thm:useful-L-p-est-on-DW}]
Since $W=1/u$,
\[  |\nabla W| = (1/u^2) |\nabla u| \leq V_{\max}^2 |\nabla u|.\]
We now interpolate between equations \eqref{eqn:pro:Hs-estimate-u-1} and  \eqref{eqn:Lsup-nabla-u}. Since $\delta = V_{\max} - V_{\min} \ll V_{\min}$, we see that
\[    \|\nabla W\|_p \leq V_{\max}^2 \|\nabla u\|_p \leq C\delta V_{\max}^{1/2-1/p} \eps^{-\frac{p-1}{p}}.\]
This proves \eqref{eqn:Lp-of-nabla-u}.
Differentiating $\nabla W$ once more, we see that
\begin{equation}
     |\Delta W| \leq 2W^3|\nabla u|^2 + W^2|\Delta u| \leq 2V_{\max}^3|\nabla u|^2 + V_{\max}^2|\Delta u|. \label{eqn:Delta-W-exact-form}
\end{equation}
Since
\begin{equation}
     \|\Delta W\|_p \leq \|\Delta W\|_\infty^{\frac{p-2}{p}} \|\Delta W\|_2^{2/p}, \label{eqn:Delta-W-split} 
\end{equation}
equations \eqref{eqn:Delta-W-exact-form}, \eqref{eqn:Lsup-nabla-u}, and \eqref{eqn:Delta-om-sup}, show that
\begin{equation}
        \|\Delta W\|_\infty \leq C\left( \delta^2 \eps^{-2} + V_{\max}\delta \eps^{-2} \right) \leq C \delta V_{\max} \eps^{-2}. \label{eqn:Delta-W-sup}
\end{equation}
Similarly, equations \eqref{eqn:Delta-W-exact-form}, \eqref{eqn:pro:Hs-estimate-u-1}, \eqref{eqn:Lsup-nabla-u}, and \eqref{eqn:Delta-om-sup} show that
\begin{align}
    \|\Delta W\|_2 \leq & \,  C( 2V_{\max}^3 \|\nabla u\|_\infty \|\nabla u\|_2 + V_{\max}^2 \|\Delta u\|_2) \\
    \leq &\,  C(\delta^2 V_{\max}^{-1/2} \eps^{-3/2} + \delta V_{\max}^{1/2} \eps^{-3/2}) \\
    \leq &\,  C\delta V_{\max}^{-1} \eps^{-3/2}. \label{eqn:Delta-W-2}
\end{align}
Combining \eqref{eqn:Delta-W-sup}, \eqref{eqn:Delta-W-2} and using \eqref{eqn:Delta-W-split}, we see that
\begin{equation}
     \|\Delta W\|_p \leq C \delta  V_{\min}^{1-4/p} \eps^{-\frac{2p-1}{p}}.  
\end{equation}
This proves \eqref{eqn:Lp-of-Delta-u}.
\end{proof}

\section{Leading order expansion of electron density} \label{sec:expansion}
 We first state a more general theorem from which Theorem \ref{thm:FD-leading-order} follows. Then, we prove Theorem \ref{thm:FD-leading-order} while delaying the proof of the more general theorem until the end of the section. 
Let 
\begin{equation}
    \mathbb{H}_c = \{ z \in \mathbb{C} : \Re z + c > 0 \}.  \label{eqn:H-c-def}
\end{equation}
We have the following general result.

\begin{theorem}\label{thm:semi-classical-leading-orders}
Let $ {2 \leq p < 3}$. Assume the following hypotheses hold.
\begin{enumerate}
    \item  Suppose that $f$ is analytic on $\mathbb{H}_c$ for some constant $c>0$ and
    \begin{equation}
       \int_{-c}^\infty |f(x+iy)| dx \lesssim O(1)  
    \end{equation}
    uniformly in $y$ for $y$ in on any compact set. 
    \item $v \geq 0$ is the sum of a {\pname} potential (see Definition \eqref{eqn:pname-def}) and an  $H^2(\Om; \R)$  function whose $H^2(\Om)$-norm is bounded by $O(\delta)$,  where $\delta = \sup v - \inf v$.
    \item  $\varphi \in H^2(\Om; \R)$  and $\|\varphi\|_{H^2(\Om)} \lesssim \delta $. 
    \item The parameters are related via
    \begin{align}
    &\eps \ll   \delta \ll v^{5/2}_{\rm max}, \\
        & v_{\rm max} \ll c.
    \end{align}
    \item $W=1/u$ denotes the Landscape potential where $u$ solves 
    \begin{equation}
        (-\epsilon^2\Delta+v)u = 1. \label{eqn:Landscape-small-v}   
    \end{equation}
\end{enumerate} 
Then,
\begin{equation}
  \den f(-\epsilon^2 \Delta + v - \varphi)  =  \frac{1}{(2\pi\epsilon)^{3}} \int_{\R^3} dp f(p^2 + W - \varphi) + \eps^{-3+1/p} \text{Rem}, \label{eqn:semi-classical-leading-orders}
\end{equation}
where 
\begin{equation}
    \|\text{Rem}\|_{L^p(\Om)} \leq C_p\int_{-Cv_{\max}}^\infty |f(z)|  \label{eqn:sc-error-est}
\end{equation}
for some $p$-dependent constant $C_p>0$.
\end{theorem}

\begin{proof}[Proof of Theorem \ref{thm:FD-leading-order}]

Let $V$ and $\phi$ be as given through the assumptions of Theorem \ref{thm:FD-leading-order}. We would like to apply Theorem \ref{thm:semi-classical-leading-orders} to both \eqref{eqn:FD-semi-classical-leading-orders} and \eqref{eqn:FD-semi-classical-leading-orders-2} simultaneously. 

We note that the Fermi-Dirac function has poles on the imaginary axis in $i\pi \Z$. Thus, we decompose $V$ as
\begin{equation}
    V =  V_{\rm cut} + v,
\end{equation}
where $V_{\min} > V_{\rm cut} \in \R > \mu$, $v > 0$, and we choose $V_{\rm cut}$ such that
\begin{equation}
    v_{\max}  = \ C^{-1} \delta^{1/4},
\end{equation}
where $C$ is the constant given in lower bound of the integral in \eqref{eqn:sc-error-est}. Consequently, we pick $f(z)$ in Theorem \ref{thm:semi-classical-leading-orders} to be
\begin{equation}
    f(z) =  f_{\rm FD}(\beta(z+V_{\rm cut}- \mu)).
\end{equation}
Thus, $\mathbb{H}_c$ is chosen with $c = V_{\rm max}- \mu$.

To prove \eqref{eqn:FD-semi-classical-leading-orders}, we apply Theorem \ref{thm:semi-classical-leading-orders} with the potential $v = V-V_{\rm cut}-\phi$ and $\varphi = 0$. To prove \eqref{eqn:FD-semi-classical-leading-orders-2}, we apply Theorem \ref{thm:semi-classical-leading-orders} with $v = V-V_{\rm cut}$ and $\varphi = \phi$. Finally, we check that the rest of the assumptions of Theorem \ref{thm:semi-classical-leading-orders} are satisfied for the above choices.

By Assumption 2 of Theorem \ref{thm:FD-leading-order} and the Sobolev inequality, $\|\varphi\|_{ \infty}\lesssim \delta$. Since $\delta \ll1$ and $V-\mu \geq C > 0$ (see Assumption 1 and 3 of Theorem \ref{thm:FD-leading-order}), we see that $V_{\rm cut} - \mu \geq O(1) > 0$. Hence, the function $f(z)=f_{\rm FD}(\beta(z + V_{\rm cut} -\mu))$ is analytic on $\mathbb{H}_{V_{\rm cut} - \mu}$ (see definition \eqref{eqn:H-c-def}) and Assumption 1 of Theorem \ref{thm:semi-classical-leading-orders} is satisfied. Clearly, items 2 and 4 of Theorem \ref{thm:semi-classical-leading-orders} are satisfied by $v$ and $\varphi$. So we consider item 3.

Note that
\begin{equation}
    \delta = v_{\max} - v_{\min} = V_{\max} - V_{\min}
\end{equation}
remains unchanged. In particular,
\begin{equation}
    \delta \ll \delta^{5/8} = v_{\rm max}^{5/2} \text{ and } v_{\rm max} \ll V_{\rm cut} - \mu.
\end{equation}
This satisfies item 3 in Theorem \ref{thm:semi-classical-leading-orders}. Item 5 of Theorem \ref{thm:semi-classical-leading-orders} can also be satisfied since $v > 0$.

It follows by Theorem \ref{thm:semi-classical-leading-orders} that the $L^2$ norm of the remainder $\text{Rem}$ (see \eqref{eqn:semi-classical-leading-orders}) is bounded above by
\begin{equation}
  \epsilon^{-3+1/2}\int_{-\delta^{1/4}}^\infty dx \, f_{\rm FD}(\beta(x+V_{\rm cut} - \mu)) \lesssim \epsilon^{-3+1/2}\beta^{-1} e^{-\beta(V_{\rm cut}-\mu-\delta^{1/4})}.
\end{equation} 
This proves the errors in \eqref{eqn:FD-semi-classical-leading-orders} and \eqref{eqn:FD-semi-classical-leading-orders-2}.
\end{proof}

The remainder of this section is devoted to the proof of Theorem \ref{thm:semi-classical-leading-orders}. 
\begin{proof}[Proof of Theorem \ref{thm:semi-classical-leading-orders}]
 First, we remark that the potential functions $v$ and $\phi$ are real and bounded. It follows that their associated Hamiltonian $-\Delta+v-\phi$ is self-adjoint (on $L^2(\R^3)$), so that the spectral theory of self-adjoint operator and its associated analytic tools apply. Moreover, the Landscape function $u$ solving \eqref{eqn:Landscape-small-v}, and the Landscape potential $W=1/u$ are also real. 

Let $f$ be a meromorphic function as given in the hypotheses of Theorem \ref{thm:semi-classical-leading-orders}. We note that
\begin{equation}
  u^{-1}(-\epsilon^2 \Delta+v)u = -\epsilon^2\Delta - 2\epsilon^2 \nabla u \cdot \nabla + u^{-1}.
\end{equation}
Let us denote
\begin{equation}
  U = -2u^{-1}\epsilon^2 \nabla u \cdot \nabla + u^{-1} - \varphi. \label{eqn:U-def}
\end{equation}
Consequently, 
\begin{equation}
  \den f(-\epsilon^2 \Delta+v -\varphi) =  \den f(-\epsilon^2\Delta+U).
\end{equation}
Since $-\epsilon^2\Delta+U$ has the same spectrum as $-\epsilon^2\Delta+v -\varphi$, we see that the spectrum of $-\epsilon^2\Delta+U$ is contained in  $[v_{\min}-O(\delta), \infty) \subset [v_{\min}/2, \infty)$, by item (2) of the assumptions in Theorem \ref{thm:semi-classical-leading-orders}. Thus, using Cauchy's theorem, we can write
\begin{equation}
    f(-\eps^2 \Delta + U) = \frac{1}{2\pi i}\int_{\G} f(z) (z-(-\eps^2 \Delta+U))^{-1},
\end{equation}
where the contour $\G$ is given in Figure \ref{fig:cauchy-int-contour-general}. 

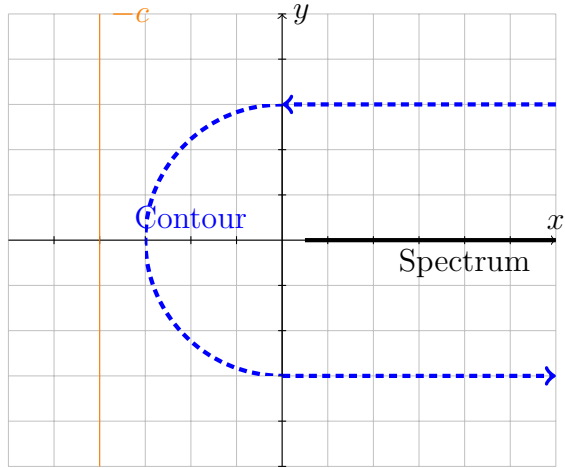
\begin{figure}[ht]
\begin{center}
\begin{tikzpicture}[scale=0.6]
    \path [draw, help lines, opacity=.5]  (-6,-5) grid (6,5);
    
    \foreach \i in {1,...,5} \draw (\i,2.5pt) -- +(0,-5pt) node [anchor=north, font=\small] {} (-\i,2.5pt) -- +(0,-5pt) node [anchor=north, font=\small] {} (2.5pt,\i) -- +(-5pt,0) node [anchor=east, font=\small] {} (2.5pt,-\i) -- +(-5pt,0) node [anchor=east, font=\small] {};
    
    \draw [->] (-6,0) -- (6,0) node [anchor=south] {$x$};
    \draw [->] (0,-5) -- (0,5) node [anchor=west] {$y$};
    
    \draw [-, color=orange] (-4,-5) -- (-4,5) node [anchor=west] {$-c$};
    
	\draw [-, ultra thick, color=black] (0.5, 0) node {} -- (6, 0) node {};
	\node[color=black] at (4,-0.5) {Spectrum};
    
	\draw [->, ultra thick, color=blue, densely dashed] (6,3) node {} -- (0,3) node {};
    \draw [->, ultra thick, color=blue, densely dashed] (0,-3) node {} -- (6,-3) node {};
    
    \begin{scope}
    \clip (-3,-3) rectangle (0,3);
    \draw [->, ultra thick, color=blue, densely dashed] (0,0) circle(3);
    \draw (0,3) -- (0,-3);
    \end{scope}

	\node[color=blue] at (-2, 0.5) {Contour};
  \end{tikzpicture}
\end{center}
\caption{We identify the complex plane $\C$ with $\R^2$ via $z = x+iy$ for $(x,y) \in \R^2$ in the diagram above. The contour $\G$ is denoted by the blue dashed line, extending to positive real infinity. The spectrum of $-\eps^2 \Delta + v-\varphi$ is contained in the solid black line. The orange line is where $\Re z = -c$ and $f(z)$ is analytic for $\Re z > -c$. \label{fig:cauchy-int-contour-general}}
\end{figure}

For simplicity, we will denote
\begin{equation}
  \oint := \frac{1}{2\pi i} \int_{\Gamma} dz \label{eqn:oint-def}
\end{equation}
for the rest of the paper. Let 
\begin{align}
    &W = 1/u \label{eqn:W-in-U-def}, \\
    & \tildeW=1/u-\varphi  \label{eqn:V-in-U-def}.
\end{align}
Then,
\begin{align}
  \den f(-\epsilon^2 \Delta + v -\varphi) =& \den f(-\epsilon^2\Delta+U) \\
  =& \den \oint f(z) R(W,\tildeW), \label{eqn:f-RWV-1}
\end{align}
where 
\begin{equation}
  R(W, \tildeW) =\left(z-(-\eps^2\Delta+U)\right)^{-1}= \Big(z-(-\epsilon^2\Delta+2\eps^2 W^{-1}\nabla W \cdot \nabla + \tildeW)\Big)^{-1}. \label{eqn:R-def}
\end{equation}
To extract leading orders and for $z \in \C$ not in the positive real line, we define,
\begin{align}
    &R:=  (z+\epsilon^2\Delta)^{-1},  \label{eq:R00}\\
    &R_R(\tildeW) := \sum_{n\geq 0} R^{n+1} \tildeW^n, \label{eqn:RR-def} \\
    &R_L(\tildeW) := \sum_{n\geq 0} \tildeW^n R^{n+1} \label{eqn:RL-def}.
\end{align}
It follows from \eqref{eqn:f-RWV-1} that 
\begin{equation}
  \den f(-\epsilon^2 \Delta + v -\varphi) =  \den \oint f(z) R_L(\tildeW) + \oint f(z) \Big(R(W,\tildeW) - R_L(\tildeW)\Big).
\end{equation}
Translation invariance of $-\Delta$ shows that for any $n$,
\[
    \den  R^n=\frac{1}{(2\pi\eps)^{3}}\int_{\R^3}(z-p^2)^{-n}\, dp.
\]
Using Cauchy's formula and Taylor's theorem, the first term $\den \oint f(z) R_L(\tildeW)$ can be computed as 
\begin{align}
    \den \oint f(z) R_L(\tildeW) =&  \sum_{n \geq 0} \oint f(z) \den  \tildeW^n R^{n+1}\\
    =& \frac{1}{(2\pi\eps)^{3}} \sum_{n \geq 0} \int_{\R^3} dp \, \oint f(z) \frac{1}{(z-p^2)^{n+1}} \tildeW^n \\
    =& \frac{1}{(2\pi\eps)^{3}} \sum_{n \geq 0} \int_{\R^3} dp \, \frac{f^{(n)}(p^2)}{n!} \tildeW^n \\
    =& \frac{1}{(2\pi\eps)^{3}} \int_{\R^3} dp \, f(p^2+\tildeW). \label{eqn:semi-classical-leading-computation}
\end{align}
Recalling that $\tildeW=W-\varphi$, this gives the leading order term in equation \eqref{eqn:semi-classical-leading-orders}. It remains to estimate the error term 
\begin{equation}
  \oint f(z) \Big(R(W,\tildeW) - R_L(\tildeW)\Big).
\end{equation}
The following Lemma is the main work horse in this estimate, whose proof is delayed until the conclusion of the proof of Theorem \ref{thm:semi-classical-leading-orders}. The Lemma involves the Schatten $p$-norm $\mathfrak{S}^p(\Om)$ given in Appendix \ref{sec:per-vol-set-up}.

\begin{lemma} \label{lem:submain-lemma}
Let $2 \leq p$. Assume that the assumptions in Theorem \ref{thm:semi-classical-leading-orders} hold and let $W$ and $\tildeW$ be given by \eqref{eqn:W-in-U-def} and \eqref{eqn:V-in-U-def}, respectively. Then
\begin{align}
   \| (1-\eps^2\Delta)(R(W,\tildeW)- R_{R}(\tildeW)) \|_{\mathfrak{S}^p(\Om)} \leq \frac{C_1\eps^{-3/p}}{C_1-\frac{v_{\max}}{d(z)}}  \delta v_{\max}^{-1/2-1/p} \eps^{1/p}, \label{eqn:R-approx-1} \\
   \| (R(W,\tildeW)- R_{L}(\tildeW))(1-\eps^2\Delta) \|_{\mathfrak{S}^p(\Om)} \leq \frac{C_1 \eps^{-3/p}} {C_2-\frac{v_{\max}}{d(z)}}   \delta v_{\max}^{-1/2-1/p}  \eps^{1/p}  \label{eqn:R-approx-2}
\end{align}
for some $C_1, C_2$ that depends on $p$ and where $d(z)$ is the distance from $z$ to the positive real line.
\end{lemma}
Assuming Lemma \ref{lem:submain-lemma}, we complete the proof of Theorem \ref{thm:semi-classical-leading-orders}. Let $ {2 \leq p < 3  }$ and $q$ be the H\"{o}lder conjugate of $p$ such that $\frac{1}{p} + \frac{1}{q} = 1$. Recalling the definition of $R_L(\tildeW)$ in \eqref{eqn:RL-def}, we may apply Lemma \ref{lem:submain-lemma} and Lemma \ref{lem:den-to-schatten} (from the Appendix) to obtain
\begin{align}
    & \hspace{-3cm}\left\|\den \oint f(z) \big( R(W,\tildeW) -  R_L(\tildeW) \big) \right\|_{L^p(\Om)} \notag \\
    =&\, \left\|\oint f(z) \den [R(W,\tildeW) - R_L(\tildeW)](1-\eps^2\Delta)(1-\eps^2\Delta)^{-1} \right\|_{\mathfrak{S}^p(\Om)} \\
    \lesssim&\,  \frac{\eps^{-3/q}}{ \inf_z d(z)} \left\|\oint f(z) [R(W,\tildeW) - R_L(\tildeW)](1-\eps^2\Delta) \right\|_{\mathfrak{S}^p(\Om)} \\
    \lesssim&\,  \frac{\eps^{-3}}{\inf_z d(z)} \int_{\Gamma} |f(z)| \left(1-C\frac{v_{\max}}{\inf_z d(z)} \right)^{-1}   \eps^{1/p}  \delta v_{max}^{-1/2-1/p} \label{eqn:ulti-1},
\end{align}
provided $v_{\max}$ is much smaller than $\inf_z d(z)$. 
Since $v_{\max} \ll c$ by item 3 in the Assumption of Theorem \ref{thm:semi-classical-leading-orders} (recall that $f(z)$ is analytic for $\Re z > -c$), we choose our contour to be such that
\begin{equation}
    c \gg d(z) = 2Cv_{\max},
\end{equation}
where $C$ is from \eqref{eqn:ulti-1}.
The assumption $\delta \ll v_{\max}^{5/2}$ proves \eqref{eqn:sc-error-est}.
\end{proof}

\begin{proof}[Proof of Lemma \ref{lem:submain-lemma}]
We will prove \eqref{eqn:R-approx-2} only. The proof for \eqref{eqn:R-approx-1} is similar. For $W$ and $\tildeW$ given in \eqref{eqn:W-in-U-def} and \eqref{eqn:V-in-U-def}, respectively,  recall from \eqref{eqn:U-def}  that
\begin{equation}
      U = 2\epsilon^2W^{-1} \nabla W \cdot \nabla  + \tildeW.
  \end{equation}
We expand the resolvent using the resolvent identity
\begin{align}
  R(W,\tildeW) =& (z-(-\epsilon^2\Delta+U))^{-1} \notag\\
  =& (z-(-\epsilon^2\Delta))^{-1} + (z-(-\epsilon^2\Delta))^{-1}U(z-(-\epsilon^2\Delta+U))^{-1} \notag\\
  =& \sum_{n \geq 0} (RU)^n R. \label{eqn:RWV-series-expand}
\end{align}
We will consider each of the $n$-th order terms separately. Let us denote
\begin{equation}
  \gamma_n = (RU)^n R \label{eqn:gamma-n-def}.
\end{equation}
Since commutator of $-\epsilon^2\Delta$ with $\tildeW$ is higher order, to leading order, we have
\begin{equation}
  \gamma_n =  \tildeW^nR^{n+1} + \text{higher order (h.o.)},
\end{equation}
where h.o. will be computed after this paragraph. Summing over $n$, to leading order,
\begin{align}
  R(W,\tildeW) =& \sum_{n \geq 0} \tildeW^nR^{n+1} + \text{h.o.} \\
    =& R_L(\tildeW) + \text{h.o.}.
\end{align}

Now we compute the higher order terms coming from $\gamma_n -  \tildeW^nR^{ n+1}$ where $\gamma_n$ is given in \eqref{eqn:gamma-n-def}. Let us introduce the following notations for clarity of exposition.
\begin{enumerate}
\item We denote 
  \begin{equation}
    W_{11} = 2W^{-1}\nabla W \label{eqn:W-11-def}
  \end{equation}
  Note that $W_{11}$ is associated with the first order derivative part of 
  \begin{equation}
      U = 2\epsilon^2W^{-1} \nabla W \cdot \nabla  + \tildeW.
  \end{equation}
  
\item We denote
  \begin{align}
    &W_{12} = -2\nabla \tildeW, \label{eqn:W12-def} \\
    &W_{21} = -\Delta \tildeW. \label{eqn:W21-def}
  \end{align}
  These terms came from the commutator
  \begin{align}
    [R, \tildeW] = R(-2\epsilon^2\nabla \tildeW \cdot \nabla + (-\epsilon^2\Delta \tildeW))R =  R(\epsilon^2  W_{11}\cdot \nabla + \eps^2 W_{21}) R \label{eqn:R-W-commutator}
  \end{align}
  when we commute $\tildeW$ in $\gamma_n$ to the left to obtain  $\tildeW^n R^{n+1}$.
\end{enumerate}
A simple way to keep track of the $W_{ij}$'s is to note that $W_{ij}$ has $i$ derivatives taken while $j$ stands for the $j$-th such quantity (in order of their introduction).

We write $U = \epsilon^2 W_{11}\cdot\nabla +\tildeW$. Then
\begin{equation}
    \gamma_n = R(\epsilon^2 W_{11}\cdot\nabla +\tildeW) R \cdots R (\epsilon^2 W_{11}\cdot\nabla +\tildeW) R.
\end{equation}
If we write $\gamma_n$ by expanding all the brackets above, we obtain
\begin{equation}
  \gamma_n =  (R\tildeW)^n R + \epsilon \sum_{i=0}^{n-1} (R\tildeW)^i (RW_{11}\cdot \epsilon \nabla) (R\tildeW)^{n-i-1} R + \gamma_n',
\end{equation}
where $\gamma_n'$ is defined by this expression and contains terms with at least two factors of $\epsilon^2 W_{11}\cdot\nabla$. By commuting $\tildeW$ to the left, we see that
\begin{equation}
    (R\tildeW)^n R = \tildeW^nR^{n+1} + \sum_{0 \leq i < j \leq n} \tildeW^{j-1} R^i [R,\tildeW] R^{j-i-1} (R\tildeW)^{n-j} R.
\end{equation}
It follows that
\begin{align}
    \gamma_n =& \tildeW^n R^{n+1} \label{eqn:gamma-expand-0}\\
        &+ \sum_{0 \leq i < j \leq n} \tildeW^{j-1} R^i [R,\tildeW] R^{j-i-1} (R\tildeW)^{n-j} R \label{eqn:commutator-term} \\
        &+ \epsilon \sum_{i=0}^{n-1} (R\tildeW)^i (RW_{11}\cdot \epsilon \nabla) (R\tildeW)^{n-i-1} R \label{eqn:W11-term} \\
        &+ \gamma_n', \label{eqn:ho-terms}
\end{align}
where we note that the leading term $\tildeW^n R^{n+1}$ was used in the computation for \eqref{eqn:semi-classical-leading-computation}. We now proceed to estimate the terms \eqref{eqn:commutator-term} -- \eqref{eqn:ho-terms} individually.

First, we estimate  the Schatten $p$-norm of commutators (see Appendix \ref{sec:per-vol-set-up}) of \eqref{eqn:commutator-term} .

\begin{lemma} \label{lem:commutator-estimate}
Let $2 \leq p$. Let $\tildeW$ be given in \eqref{eqn:V-in-U-def}. Then
\begin{equation}
    \|[R, \tildeW]\|_{\mathfrak{S}^p(\Om)} \leq C\frac{\eps^{-3/p}}{d(z)^2}    \delta v_{\max}^{1/2-1/p} \eps^{1/p} \label{eqn:lem:commutator-estimate}.
\end{equation}
where $d(z)$ is the distance from $z$ to the positive real line.
\end{lemma}
\begin{proof}
We compute
\begin{equation}
    [R, \tildeW] = R(\eps^2 W_{11}\cdot\nabla + \epsilon^2 W_{21})R.
\end{equation}
Kato-Seiler-Simon inequality shows that
\begin{align}
    \|[R, \tildeW]\|_{\mathfrak{S}^p(\Om)} \leq & \epsilon \|RW_{12}\|_{\mathfrak{S}^p(\Om)} \|\eps\nabla R\|_{\mathfrak{S}^\infty(\Om)} + \|R\|_{\mathfrak{S}^\infty(\Om)} \eps^2 \|W_{21}R\|_{\mathfrak{S}^p(\Om)} \\
    \leq & \frac{C\eps^{-3/p}}{d(z)^2} (\eps\|W_{12}\|_{p } + \eps^2\|W_{21}\|_{p }).
\end{align}
Recalling definitions \eqref{eqn:W12-def} and \eqref{eqn:W21-def}, we see that
\begin{align}
    \|W_{12}\|_{p} &\leq C\|\nabla \tildeW\|_{p}, \\
    \|W_{21}\|_{p} &\leq C\|\Delta \tildeW\|_{p}.
\end{align}
By definition \eqref{eqn:V-in-U-def} of $\tildeW$, Theorem \ref{thm:useful-L-p-est-on-DW} proves \eqref{eqn:lem:commutator-estimate}.
\end{proof}

Applying Lemma \ref{lem:commutator-estimate} to equation \eqref{eqn:commutator-term}, we see that
\begin{align}
    \left\|\sum_{0 \leq i < j \leq n} \tildeW^{j-1} R^i [R,\tildeW] R^{j-i-1} (R\tildeW)^{n-j} R (1-\eps^2\Delta)\right\|_{\mathfrak{S}^p(\Om)} \notag \\
    \leq  \frac{\eps^{-3/p}n^2 C^n}{d(z)^{n+1}} \|\tildeW\|_{\infty}^{n-1}   \delta v_{\max}^{1/2-1/p} \eps^{1/p}. \label{eqn:commutator-est-end}
\end{align}
This concludes the estimate for \eqref{eqn:commutator-term}.  

We can estimate \eqref{eqn:W11-term} in the spirit of the estimate \eqref{eqn:commutator-term} by using an analogue of \eqref{eqn:commutator-est-end}. We obtain
\begin{align}
    \left\|\epsilon \sum_{i=0}^{n-1} (R\tildeW)^i (RW_{11}\cdot  \epsilon \right. & \left. \nabla) (R\tildeW)^{n-i-1} R  (1-\eps^2\Delta) \vphantom{\sum_{i=0}^{n-1}}\right\|_{\mathfrak{S}^p(\Om)} \notag \\
    &\leq  \frac{\eps^{-3/p}nC^n}{d(z)^{n+1}} \|\tildeW\|_{ \infty}^{n-1}  \eps \|W_{11}\|_{ p} \\
    &\leq  \frac{\eps^{-3/p}nC^n}{d(z)^{n+1}} \|\tildeW\|_{ \infty}^{n-1} \delta v_{\max}^{1/2-1/p} \eps^{1/p}. \label{eqn:W11-term-est-end}
\end{align}

Finally, we estimate \eqref{eqn:ho-terms}. The term $\gamma'_n$ consists of all possible terms of the form
\begin{equation}
    RX_1 R X_2 \cdots R X_n R,
\end{equation}
where $X_i$ is one of $\tildeW$ or $\eps W_{11}\cdot (\eps \nabla)$ with at least two of the latter factor. Without loss of generality, we assume that $X_1 = \eps W_{11}\cdot (\eps \nabla)$. By Lemma \ref{lem:den-to-schatten},  Theorem \ref{thm:useful-L-p-est-on-DW} and Proposition  \ref{prop:sup-norm-grad-u}, we see that
\begin{align}
     \| RX_1 R X_2&  \cdots R X_n R  (1-\eps^2\Delta) \|_{\mathfrak{S}^p(\Om)} \notag\\
    \leq& \frac{\eps^{-3/p}C}{d(z)} \|R\eps W_{11}\|_{\mathfrak{S}^p(\Om)} \|\eps \nabla R X_2\|_{\mathfrak{S}^\infty(\Om)} \prod_{i\geq3} \|R X_i\|_{\mathfrak{S}^\infty(\Om)} \\
    \leq & \frac{\eps^{-3/p}C^n}{d(z)^{n+1}}   (\|\tildeW\|_\infty + \eps\|u\|_{ \infty}\|\nabla W\|_\infty)^{n-1}   \|\eps\nabla W\|_{ p} \|u\|_{ \infty} \\
    \leq& \frac{\eps^{-3/p}C^n}{d(z)^{n+1}}   (\|\tildeW\|_\infty + \eps v_{\min}^{-1}\|\nabla W\|_\infty)^{n-1}   \|\eps\nabla W\|_{ p} v_{\min}^{-1}. \label{eqn:gamma-prime-est}
\end{align}
Combining \eqref{eqn:commutator-est-end} (for \eqref{eqn:commutator-term}), \eqref{eqn:W11-term-est-end}  (for \eqref{eqn:W11-term}), \eqref{eqn:gamma-prime-est} (for \eqref{eqn:ho-terms}), and using the binomial theorem (or a modification thereof), the fact $\delta \ll v_{\min}$ and Theorem \ref{thm:useful-L-p-est-on-DW}, we see that
\begin{align}
    \left\|(\gamma_n - \tildeW^n R^{n+1})\right. & \left. (1-\eps^2\Delta)  \vphantom{(\gamma_n - \tildeW^n R^{n+1})}\right\|_{\mathfrak{S}^p(\Om)} \notag \\
    \leq & \frac{\eps^{-3/p} C^n}{d(z)^{n+1}} (v_{\max} +   \delta v_{\min}^{-1/2} )^{n-1}  \delta v_{\max}^{-1/2-2/p} \eps^{1/p}.
\end{align}
Together with \eqref{eqn:RWV-series-expand}, \eqref{eqn:R-approx-2} and Lemma \ref{lem:submain-lemma} is now proved.
\end{proof}

\section{Consequence of the integrability condition} \label{sec:integrability}

Before we move on to the proof of the main result Theorem \ref{thm:eff-eqn}. We dedicate this short section to elucidate the implied relationships between different parameters $\eps, \delta, \beta$, and $\mu$. We achieve this through the integrability of the Poisson equation \eqref{eqn:integrability-cond}, which we now recall and elaborate.

Let $F = F_{\rm REHF}$, $F_{\rm PL}$ or $ F_{\rm LSC}$ (see \eqref{eqn:F-REHF-def} -- \eqref{eqn:F-LSC-def}). Integrating the left and right hand sides of \eqref{eqn:general-F-def}, we obtain an equation of the form
\begin{align}
     \kappa_0 := \frac{1}{|\Om|}\int_{\Om} \kappa = \frac{1}{|\Om|}\int_\Om F  \label{eqn:mu-0-def-all-F}.
\end{align}
The goal of this section is to prove bounds on $\mu$ given that \eqref{eqn:mu-0-def-all-F} holds.

\begin{lemma} \label{lem:mu-estimate}
Let $V$ be a bounded potential and $\kappa_0 \in \R > 0$ and  $\|\phi\|_{H^2(\Om)} \lesssim \delta$.  Let Assumptions \ref{ass:semi-classical} -- \ref{ass:doping}, and \eqref{eqn:mu-0-def-all-F} hold. Assume also that \eqref{eqn:mu-0-def-all-F} holds for any $F$ given in \eqref{eqn:F-REHF-def} -- \eqref{eqn:F-LSC-def}. 
Then,
\begin{equation}
    0 < \mu \text{ and } V-\mu  > K .
\end{equation}
\end{lemma}
\begin{proof}
We consider the special case where $F= F_{\rm LSC}$, all other cases follow from Theorem \ref{thm:FD-leading-order}. In this case,  Assumption \ref{ass:doping} implies 
\begin{equation}
     \kappa_0 = \frac{1}{(2\pi\eps)^3 |\Om|} \int_\Om dx \, \int_{\R^3} dp f_{\rm FD}(\beta(p^2+W - \phi + V_{\rm cut} - \mu)) \label{eqn:mu-0-def},
\end{equation}
where $u=1/W$ solves
\begin{equation}
    (-\epsilon^2\Delta+V-V_{\rm cut})u = 1.
\end{equation}
Let $\theta(\mu)$ denote the right hand side of \eqref{eqn:mu-0-def}. We first note that $\theta(\mu)$ is increasing in $\mu$ since $f_{\rm FD}$ is a decreasing function. So it suffices for us to check that $\theta(0) < \kappa_0 < \theta(V_{\min} - K)$.    For a generic $\eta \in \R$,
\begin{align}
 \int_{\R^3} dp \, f_{\rm FD}(&\beta p^2 + \beta (W - \phi + V_{\rm cut}-\eta)) \\
    =& \frac{4\pi}{\beta^{3/2}} \int_0^\infty dq \,q^2 f_{\rm FD}(q^2 + \beta (W - \phi + V_{\rm cut}-\eta)).
\end{align}
Since $\|\phi\|_{H^2} \lesssim \delta$, we can find a constant $C$ such that
\begin{align}
    &0 < V_{\min} \\
    &0 <  K - C\delta \leq (W - \phi + V_{\rm cut}) - (V_{\rm min} - K)
\end{align}
Thus, we see that
\begin{align}
    & f_{\rm FD}(q^2 + \beta (W - \phi + V_{\rm cut}-0)) \approx e^{-q^2}  e^{ -\beta (W-\phi+V_{\rm cut})}, \\
    & f_{\rm FD}(q^2 + \beta ((W - \phi + V_{\rm cut})-(V_{\min}-K))) \approx e^{-q^2} e^{ -\beta (W - \phi + V_{\rm cut}-V_{\min}+K)},
\end{align}
where we recall that $A \approx B$ means $A \lesssim B \lesssim A$. It follows that
\begin{align}
    & \theta(0)  \lesssim \eps^{-3} \beta^{-3/2} e^{-\beta (V_{\min}-C\delta)}, \\
    & \theta(V_{\min}-K)  \gtrsim  \eps^{-3} \beta^{-3/2} e^{-\beta (K+C\delta)}
\end{align}
for some constant $C> 0$. Thus, Assumption \ref{ass:temperature} shows that a solution $0 < \mu_0 < V_{\min}-K$  of $\theta(\mu_0)=\kappa_0$  exists by  Assumptions \ref{ass:semi-classical} and \ref{ass:pname}.

\end{proof}

\begin{corollary} \label{cor:kappa-0-fixed-scale}
Let Assumptions \ref{ass:semi-classical} -- \ref{ass:doping} hold. Assume also that \eqref{eqn:mu-0-def-all-F} holds and $\|\phi\|_{H^2(\Om)} \lesssim \delta$. Then
\begin{equation}
    \eps^{-3}\beta^{-3/2} e^{-\beta(V_{\max} - \mu+C\delta)} \lesssim 1 \lesssim \eps^{-3}\beta^{-3/2} e^{-\beta(V_{\min} - \mu-C\delta)} \label{eqn:kappa-fixed-scale-est}
\end{equation}
for some constant $C > 0$.
\end{corollary}
\begin{proof}
By Theorem \ref{thm:FD-leading-order}, it suffices for us to assume that \eqref{eqn:mu-0-def} hold. Lemma \ref{lem:mu-estimate} shows that $V-\mu > K > 0$. Since $\beta$ is large, the Fermi-Dirac distribution $f_{\rm FD}(\beta(p^2 + V-\mu))$ is well approximated by $e^{-\beta p^2} e^{-\beta(V-\mu)}$. Integrating $dp$, we have that
\begin{equation}
    \int_{\R^3} dp e^{-\beta p^2} \approx \beta^{-3/2},
\end{equation}
where we recall that $A \approx B$ means $A \lesssim B \lesssim A$. Moreover,
\begin{equation}
    e^{-\beta(V_{\max} - \mu+C\delta)} \leq e^{-\beta(W-\phi+V_{\rm cut}-\mu)} \leq e^{-\beta(V_{\min} - \mu-C\delta)}.
\end{equation}
It follows by \eqref{eqn:mu-0-def} that \eqref{eqn:kappa-fixed-scale-est} is proved.
\end{proof}

\section{Proof of the main result: Theorem \ref{thm:eff-eqn}} \label{sec:main-result-proof}
\begin{proof}[Proof of Theorem \ref{thm:eff-eqn}]
Recall that we can write the REHF, PL, and LSC equations (see \eqref{eqn:general-F-def}) in the form
\begin{equation}
    -\Delta \phi = \kappa - F(\phi, \mu),
\end{equation}
where $F$ is one of \eqref{eqn:F-REHF-def} -- \eqref{eqn:F-LSC-def}.

For a fixed choice of X = REHF, PL, or LSC, let $(\phi_0, \mu)$ denote a solution of equation X satisfying the assumptions of Theorem \ref{thm:eff-eqn}. We look for a solution, $\phi$, of the corresponding equation Y =  REHF, PL, and LSC, Y $\not=$ X, near $(\phi_0, \mu)$ of the form $\phi = \phi_0 + \varphi$. Substituting this ansatz into \eqref{eqn:general-F-def} of the Y equation, we obtain
\begin{align}
    -\Delta \phi_0 - \Delta \varphi =& \kappa - F_Y(\phi_0+\varphi, \mu) \\
        =& \kappa - F_Y(\phi_0, \mu) + F_Y(\phi_0, \mu) - F_Y(\phi_0+\varphi, \mu).
\end{align}
Rearranging, we obtain
\begin{equation}
    - \Delta \varphi =  \kappa' + F_Y(\phi_0, \mu) - F_Y(\phi_0+\varphi, \mu), \label{eqn:eff-eqn-step-1}
\end{equation}
where 
\begin{align}
    \kappa' =& \kappa - F_Y(\phi_0, \mu)+\Delta \phi_0 \\
    =& F_{\rm X}(\phi_0, \mu) - F_{\rm Y}(\phi_0, \mu).
\end{align}
Theorem \ref{thm:FD-leading-order}, Corollary \ref{cor:kappa-0-fixed-scale}, and the scaling in Assumption \ref{ass:pname} shows that
\begin{align}
    \|\kappa'\|_{L^2(\Om)} 
    \lesssim &\,  \eps^{-3+1/2}\beta^{-1} e^{-\beta(V_{\rm cut}-\mu-\delta^{1/4})} \\
    \lesssim & \,  \eps^{1/2}   e^{-C\beta \delta^{1/4}}   \\
    \lesssim &\,  \eps^{1/2- C\delta^{1/4}} \label{eqn:kappa-prime-est}
\end{align}
for some constant $C = O(1)$, independent of $\eps$ and $\delta$.
Let $M$ denote the G\^{a}teaux derivative of $F$ at $\phi_0$: \begin{equation}
    M = d_\phi F(\phi, \mu) \mid_{\phi=\phi_0} . \label{eqn:M-def-primitive}
\end{equation}
We see that \eqref{eqn:eff-eqn-step-1} can be write as
\begin{equation}
    (-\Delta+M)\varphi = \kappa' + N(\varphi) ,\label{eqn:N-def}
\end{equation}
where $N$ is defined by this expression. Let us denote 
\begin{equation}
    L = -\Delta+M . \label{eqn:linear-operator}  
\end{equation}
The rest of the analysis rests upon the following abstract lemma and subsequent theorems.

\begin{lemma}[Main Lemma] \label{lem:main-lemma}
Let ${\mathcal H}_1$ and ${\mathcal H}_2$ be two Hilbert spaces such that ${\mathcal H}_1 \subset {\mathcal H}_2$ is dense (in the $\mathcal{H}_2$ topology). Let $L$ be an operator on ${\mathcal H}_2$ with domain ${\mathcal H}_1$ and $N$ be a function on ${\mathcal H}_1$ with range in ${\mathcal H}_2$. Assume that $L$ is invertible on ${\mathcal H}_2$ and there is a $0 < m \in \R$ such that
\begin{equation}
    \|L^{-1} \|_{{\mathcal H}_2 \rightarrow {\mathcal H}_1} \leq m^{-1}, \label{eqn:L-lower-ass}
\end{equation}
and
\begin{equation}
    \|N(\phi_1) - N(\phi_2)\|_{{\mathcal H}_2} < C_N (\|\phi_1\|_{{\mathcal H}_1} + \|\phi_2\|_{{\mathcal H}_1})\|\phi_1 - \phi_2\|_{{\mathcal H}_1} \label{eqn:nonlin-ass}
\end{equation}
for some constant $C_N$ on a ball of radius at least $C m^{-1} \|\kappa'\|_{{\mathcal H}_2}$ centered the origin for some constant $C > 0$. Let $\kappa' \in {\mathcal H}_2$. If
\begin{align}
    & \|\kappa'\|_{{\mathcal H}_2} \ll m \text{ and } \label{eqn:x-0-ass}\\
    & C_N \|\kappa'\|_{{\mathcal H}_2} \ll m^2, \label{eqn:param-scaling-ass} 
\end{align}
then there exists a unique solution $\varphi$ on the set
\begin{equation}
    \left\{ \varphi \in \mathcal{H}_1 : \|\varphi\|_{\mathcal{H}_1} \leq \frac{1}{100} m C_N^{-1} \right\}
\end{equation} 
to the equation
\begin{equation}
    L\varphi = \kappa' + N(\varphi).
\end{equation}
Moreover,
\begin{equation}
    \|\varphi\|_{{\mathcal H}_1} \lesssim m^{-1} \|\kappa'\|_{{\mathcal H}_2} .
\end{equation}
\end{lemma}
\begin{proof}
This is just the implicit function theorem with explicit estimates written out.  See, for example, Chapter XIV of \cite{Lang}. 
\end{proof}

Let $L$ be given by \eqref{eqn:linear-operator} for either one of $F = F_{\rm REHF}, F_{\rm PL}$, or $ F_{\rm LSC}$ and
\begin{equation}
    m_0 = \eps^{\delta^{1/4}} \label{eqn:m-0-def}.
\end{equation}
\begin{theorem} \label{thm:L-lower-bound-tot}
Let the assumptions of Theorem \ref{thm:eff-eqn} hold. Then $L$ is bounded below on $L^2(\Om)$:
\begin{equation}
    \| L f\|_{L^2(\Om)}  \geq C_1\|(-\Delta + m_0^{C_2})f\|_{L^2(\Om)} \label{eqn:L-lower-bound-tot}
\end{equation}
for some constant $C$ and any $f \in H^2(\Om)$.
\end{theorem}

Let $N$ be defined via \eqref{eqn:N-def} for $F$ being any one of $F_{\rm REHF}, F_{\rm PL}$, or $ F_{\rm LSC}$. Then we have the following result.

\begin{theorem} \label{thm:nonlin-tot}
Let the assumptions of Theorem \ref{thm:eff-eqn} hold. The nonlinear operator $N$ has the following estimate
\begin{equation}
	\|N(\phi_1) - N(\phi_2)\|_{L^2(\Om)} \leq C_3 m_0^{-C_4} (\|\phi_1\|_{H^1(\Om)} + \|\phi_2\|_{H^1(\Om)}) \|\phi_1 - \phi_2\|_{L^2(\Om)} \label{eqn:nonlin-tot}
\end{equation}
for $\phi_1$ and $\phi_2$ in $H^1(\Om)$ provided $\|\phi_i\|_{H^1(\Om)} \lesssim m_0^{C_3}$ for some large constant $C_3$, where $m_0$ is given in \eqref{eqn:m-0-def}.
\end{theorem}

Theorems \ref{thm:L-lower-bound-tot} and \ref{thm:nonlin-tot} are proved in Sections \ref{sec:lin-anal} and \ref{sec:nonlin-anal} below, respectively. Section \ref{sec:integrability} provides some preliminary estimates on parameters $\epsilon,\beta, \mu$ etc. due to the integrability condition \eqref{eqn:integrability-cond}.

Now we apply Lemma \ref{lem:main-lemma} to \eqref{eqn:N-def}.
We take ${\mathcal H}_1 = H^2(\Om)$ and ${\mathcal H}_2 = L^2(\Om)$. By Theorem \ref{thm:L-lower-bound-tot}, the linear estimate \eqref{eqn:L-lower-ass} of Lemma \ref{lem:main-lemma} is satisfied with
\begin{equation}
    m = C_1\epsilon^{C\delta^{1/4}}
\end{equation}
for some constant $C_1$ and $C_2$ given in \eqref{eqn:L-lower-bound-tot}. Moreover, Theorem \ref{thm:nonlin-tot} shows that $C_N$ of Lemma \ref{lem:main-lemma} can be taken to be
\begin{equation}
    C_N = C_3 \eps^{-C_4\delta^{1/4}}
\end{equation}
where $C_3$ and $C_4$ are constants given in \eqref{eqn:nonlin-tot}. 
Together with equation \eqref{eqn:kappa-prime-est}, we see that
\begin{equation}
    \|\kappa'\| \leq C_N \|\kappa'\|_{L^2(\Om)} \lesssim C_1 \eps^{1/2- C_3\delta^{1/4}} \ll \min(m, m^2)
\end{equation}
by Assumption \ref{ass:pname}. This proves \eqref{eqn:param-scaling-ass} of Lemma \ref{lem:main-lemma}. Consequently, Theorem \ref{thm:eff-eqn} is proved by Lemma \ref{lem:main-lemma}.  We remark that the reality of $\phi$ is established by the complex conjugation symmetry \eqref{eqn:cc-sym} of \eqref{eqn:general-F-def} and the uniqueness of solution from the above fixed point argument. 
\end{proof}

\section{Linear Analysis} \label{sec:lin-anal}
In this section we prove Theorems \ref{thm:L-lower-bound-tot} in three parts: in each of the following subsections, we prove a version of Theorem \ref{thm:L-lower-bound-tot} for the case of REHF, PL, and LSC in Theorems \ref{thm:L-lower-bound}, \ref{thm:L-lower-bound-PL}, and  \ref{thm:L-lower-bound-LSC}, respectively.

\subsection{Proof of Theorem \ref{thm:L-lower-bound-tot}: REHF case}
Let
\begin{equation}
    M_{\rm REHF} = d_\phi F_{\rm REHF} \mid_{\phi = \phi_0}
\end{equation}
be the G\^ateaux derivative of $F_{\rm REHF}(\cdot, \mu)$ at $(\phi_0,\mu)$ (cf. \eqref{eqn:M-def-primitive}). Recall that $L_{\rm REHF} = -\Delta+ M_{\rm REHF}$ and that $m_0$ is defined by \eqref{eqn:m-0-def}.

\begin{theorem} \label{thm:L-lower-bound}
 Let the assumptions of Theorem \ref{thm:eff-eqn} hold.  Then $L_{\rm REHF}$ is a positive self-adjoint operator on $L^2(\Om)$ and
\begin{equation}
    L_{\rm REHF} \gtrsim -\Delta + m_0^C \label{eqn:L-lower-bound}
\end{equation}
for some constant $C$.
\end{theorem}

We begin by recording a few auxiliary lemmas first.

\begin{lemma} \label{lem:M-basic-form}
Let the assumptions of Theorem \ref{thm:eff-eqn} hold. For any $\varphi \in L^2(\Om)$, we have
\begin{equation}
     M_{\rm REHF} \varphi = -\den \oint f_{\rm FD}(\beta(z + V_{\rm cut} -\mu)) (z-h)^{-1}\varphi(z-h)^{-1}, \label{eqn:M-def}
\end{equation}
where
\begin{equation}
    h = -\eps^2 \Delta + V - V_{\rm cut} - \phi_0. \label{eqn:h-h-0-def}
\end{equation}
Moreover, $M_{\rm REHF}$ is a bounded positive self-adjoint operator on $L^2(\Om)$ with
\begin{equation}
    \|M_{\rm REHF}\|_{\mathfrak{S}^\infty(\Om)} \lesssim m_0^{-C} . \label{eqn:M-upper-bound} 
\end{equation}
for some constant $C$.
\end{lemma}
\begin{proof}[Proof of Lemma \ref{lem:M-basic-form}]
We will only prove \eqref{eqn:M-upper-bound}. The rest of the properties are proved in \cite{CS}.  We remark that $h$ is self-adjoint on $L^2(\R^3)$ since each of the potential functions $V$ and $\phi_0$ is real.  Let $f,g \in L^2(\Om)$. Let $\Tr_\Om$ denote the trace per volume $\Om$ operator (see Appendix \ref{sec:per-vol-set-up}):
\begin{equation}
    \Tr_\Om A = \frac{1}{|\Om|}\Tr \chi_\Om A,
\end{equation}
where $A$ is an operator on $L^2(\R^3)$ and $\chi_\Om$ is the indicator function of $\Om$. Since $V-V_{\rm cut}-\phi_0$ is bounded, we see that $h$ is self-adjoint. Moreover,
\begin{equation}
    \|(1-\epsilon^2\Delta)(z-h)^{-1}\|_{\mathfrak{S}^\infty(\Om)} \lesssim \left(1+\frac{1}{d(z)}\right), \label{eqn:change-of-resolvent}
\end{equation}
where $d(z)$ is the distance from $z$ to the contour, $\Gamma$ (see Figure \ref{fig:cauchy-int-contour-general} for definition of $\Gamma$), of integration in 
\begin{equation}
    \oint = \frac{1}{2\pi i}\int_\Gamma dz. 
\end{equation}
Let $\mathfrak{S}^p(\Om)$ denote the standard Schatten norm associated to $\Tr_\Om$ (see Appendix \ref{sec:per-vol-set-up}). By \eqref{eqn:change-of-resolvent}, the definition of $\den$ (rigorously defined via \eqref{eqn:den-def-via-riesz}), and the Kato-Seiler-Simon inequality, we see that
\begin{align}
    &\hspace{-1cm} \frac{1}{|\Om|}\lan g, M_{\rm REHF}f \ran_{L^2(\Om)} \\
    =&\, \oint f_{\rm FD}(\beta(z-\mu)) \Tr_{\Om} \bar g (z-h)^{-1} f (z-h)^{-1}  \label{eqn:M-tracial-char} \\
    \leq & \,  \oint  |f_{\rm FD}(\beta(z-\mu))| \|g (z-h)^{-1} \|_{\mathfrak{S}^2(\Om)} \|f (z-h)^{-1} \|_{\mathfrak{S}^2(\Om)} \notag \\
    \leq & \,  \oint  \left| 1+\frac{1}{d(z)} \right|^2 |f_{\rm FD}(\beta(z-\mu))| \|g\|_{L^2(\Om)}\|f\|_{L^2(\Om)} \|(1-\epsilon^2\Delta)^{-1}\|_{\mathfrak{S}^2(\Om)}^2 \notag \\
    \leq & \,  \epsilon^{-3}\oint  \frac{|f_{\rm FD}(\beta(z-\mu))|}{d(z)^2} \|g\|_{L^2(\Om)}\|f\|_{L^2(\Om)}. \notag
\end{align}
Since $f_{\rm FD}(\beta(z-V_{\rm cut} - \mu))$ is holomorphic on $\{ z : \Re z > \mu - V_{\rm cut} \}$, we may choose the contour such that $|d(z)| = O(\|V_1\|_\infty) = O(\delta^{1/4})$ (see \eqref{eqn:cut-size}) and
\begin{equation}
     \lan g, M_{\rm REHF} f \ran_{L^2(\Om)} \lesssim    \beta^{-1}\delta^{-1/2}\epsilon^{-3} e^{-\beta(V_{\rm cut} - \mu - \delta^{1/4})}\|g\|_{L^2(\Om)}\|f\|_{L^2(\Om)},
\end{equation}
where we note that the additional factor of $\beta^{-1}$ came from integration in $z$. By Corollary \ref{cor:kappa-0-fixed-scale}, 
we see that \eqref{eqn:M-upper-bound} follows. This proves the $\mathfrak{S}^\infty(\Om)$ bound for $M_{\rm REHF}$. One can see that $M_{\rm REHF}$ is self-adjoint by using the tracial characterization \eqref{eqn:M-tracial-char} and the cyclicity of trace.
\end{proof}

Let
\begin{equation}
    |\nabla| = \sqrt{-\Delta}, \label{eqn:abs-nabla-def}
\end{equation}
where the square root is taken via the Borel functional calculus under periodic boundary condition on $\Om$. The following lemma is crucial to our linear analysis and is based on unpublished notes of Chenn and I. M. Sigal, and proved in Lemma 6 of \cite{CZ}  with $\eps=1$ . 

\begin{lemma} \label{lem:L0-explicit-compute}

Assume that $v>0$. 
\begin{multline}\label{eqn:MexplicitForm}
   -\den\oint dz\, f_{\rm FD}(\beta(z+v))(z+\eps^2\Delta)^{-1}\varphi(z+\eps^2\Delta)^{-1}=\\
    \frac{1}{8\pi^2 \eps^3 }  \int_{0}^\infty dt \, f_{\rm FD}(\beta(t+v))   \frac{1}{ |\epsilon \nabla|}\log \left( \left| \frac{\sqrt{4t} +   |\eps\nabla|}{\sqrt{4t} -  |\eps\nabla|} \right| \right)\, \varphi  .
\end{multline}
\end{lemma}

In view of Lemma \ref{lem:L0-explicit-compute}, define 
\begin{equation}
    M_{\rm sc}   =  \frac{1}{8\pi^2 \eps^3} \int_{0}^\infty   dt\,  f_{FD}(\beta(t+W+V_{\rm cut}-\phi_0-\mu)) \frac{1}{ |\epsilon \nabla|}\log \left( \left| \frac{\sqrt{4t} +   |\epsilon \nabla|}{\sqrt{4t} -  |\epsilon \nabla|} \right| \right) , \label{eqn:M-sc-def}
\end{equation}
where $W=1/u$ and
\begin{equation}
    (-\epsilon^2\Delta+V-V_{\rm cut})u=1. \label{eqn:W-for-V-V-cut}
\end{equation}
 We remark that when  $M_{\rm sc}$ acts on the zero-($0$-)eigenvectors of $|\nabla|$ (i.e. constants), it is assumed that the integrand in \eqref{eqn:M-sc-def} is interpreted as 
\begin{equation}
    \frac{1}{0}\log \left( \left| \frac{\sqrt{4t} +  0}{\sqrt{4t} - 0} \right| \right) = \frac{1}{\sqrt{t}} \label{eqn:0-p-interpretation}
\end{equation}
for $t > 0$, to ensure continuity of the integrand.  Finally, we recall that $m_0 = \epsilon^{\delta^{1/4}}$ was defined in \eqref{eqn:m-0-def}. 

\begin{lemma} \label{lem:mass-sc}
Let the assumptions of Theorem \ref{thm:eff-eqn} hold. Let $p,q \geq 2$ satisfy $\frac{1}{p} +\frac{1}{q} = \frac{1}{2}$, then
\begin{equation}
      \left\| M_{\rm REHF}f - M_{\rm sc} f \right\|_{L^2(\Om)} \lesssim \epsilon^{1/p}  m_0^{-C} \| f\|_{L^q(\Om)} \label{eqn:M-sc-explicit-form}
\end{equation}
for some constant $C$. Consequently, picking $p=2$ and $q = \infty$, and by Sobolev's inequality,
\begin{equation}
      \left\| M_{\rm REHF} f - M_{\rm sc} f \right\|_{L^2(\Om)} \lesssim \epsilon^{1/2}  m_0^{-C}  \| f\|_{H^2(\Om)}.
\end{equation}
\end{lemma}
\begin{proof}
We start with equation \eqref{eqn:M-def}. Let $u$ solve the shifted Landscape equation \eqref{eqn:W-for-V-V-cut}. Conjugating inside $\den$ by $u$, we obtain
\begin{equation}
    M_{\rm REHF} f = -\den\oint f_{\rm FD}(\beta(z+V_{\rm cut}-\mu)) R(W,\tildeW)fR(W,\tildeW),
\end{equation}
where $W=1/u$,  $\tildeW =1/u - \phi_0$  and $R(W,\tildeW)$ is defined in \eqref{eqn:R-def}. Similarly, we recall the definition of $R_R(\tildeW)$ and $R_L(\tildeW)$ in \eqref{eqn:RR-def} and \eqref{eqn:RL-def}, respectively. We write
\begin{equation}
     R_{L}' := R(W,\tildeW) - R_{L}(\tildeW),  \label{eqn:R-prime-RL-def}
\end{equation}
and define $R_R'$ similarly. 
Then, we may rewrite $M_{\rm REHF} f$ as
\begin{align}
    M_{\rm REHF} f = -\den\oint & f_{\rm FD}(\beta(z+V_{\rm cut}-\mu)) R_L(\tildeW)fR_R(\tildeW) \label{eqn:M-sc-primitive-def}  \\
    &- \den\oint f_{\rm FD}(\beta(z+V_{\rm cut}-\mu))  [R'_L f R(W, \tildeW) \label{eqn:M-higher-1} \\
    &\hspace{3cm} + R(W, \tildeW)f R_R' + R'_L f R'_R]. \label{eqn:M-higher-2}
\end{align}

We first consider the leading order term on right hand side of \eqref{eqn:M-sc-primitive-def}.
\begin{equation}\label{eq:Mlead}
   M_{\rm lead} :=-\den \oint f_{\rm FD}(\beta(z+V_{\rm cut}-\mu)) R_L(\tildeW)fR_R(\tildeW).
\end{equation} 
On $L^2(\R^3)$, the operators $(z+\epsilon^2\Delta)^{-n}$ have integral kernels
\begin{equation}
    \frac{1}{(2\pi)^3} \int_{\R^3} dp \,  \frac{1}{(z-\epsilon^2 p^2)^n} e^{ip(x-y)}. \label{eqn:integral-kernel-of-n-fold-resolvent}
\end{equation}
Inserting \eqref{eqn:integral-kernel-of-n-fold-resolvent} into  \eqref{eq:Mlead}, we see that 
\begin{multline}
      (M_{\rm lead}f)(x) =  \frac{1}{(2\pi)^6} \oint dz  \int_{\R^3 \times \R^3\times \R^3} dp dy dq \sum_{n, m \geq 1} f_{\rm FD}(\beta(z+V_{\rm cut}-\mu))  \\
     \times \frac{\tildeW^{n+m}(x)f(y)}{(z-\epsilon^2 p^2)^n(z-\epsilon^2 q^2)^m}  e^{i(p-q)(x-y)} .
\end{multline}
For any real numbers $A, B$, we note that
\begin{equation}
    \frac{d^k}{dz^k} \frac{1}{(z-A)(z-B)} = (-1)^k k!\sum_{n+m = k; \ 1 \leq m,n}\frac{1}{(z-A)^n(z-B)^m}.
\end{equation}
It follows by Taylor's theorem that
\begin{align*}
    (M_{\rm lead}f)(x) =& \frac{1}{(2\pi)^6}\oint dz \int_{\R^3 \times \R^3\times \R^3} dp dy dq \sum_{k \geq 0}  f_{\rm FD}(\beta(z+V_{\rm cut}-\mu)) \\
    &\times \frac{(-1)^{k} \tildeW^{k}(x)}{k!} \frac{d^k}{dz^k} \frac{f(y)}{(z-\epsilon^2  p^2)(z-\epsilon^2 q^2)} e^{i(p-q)(x-y)}\\
    =& \frac{1}{(2\pi)^6} \oint dz \int_{\R^3 \times \R^3\times \R^3} dp dy dq f_{\rm FD}(\beta(z+V_{\rm cut}-\mu)) \\
    &\times  \frac{f(y)}{(z-\epsilon^2 p^2-\tildeW(x))(z-\epsilon^2 q^2-\tildeW(x))}e^{i(p-q)(x-y)}.
\end{align*}
 
By Fourier transforming back to the position basis, we have that
\begin{equation}
     (M_{\rm lead}f)(x) = \oint dz f_{\rm FD}(\beta(z+V_{\rm cut}-\mu)) \den[(z-h(x))^{-1}f(z-h(x))^{-1}](x),
\end{equation}
where
\begin{equation}
    h(x) = -\eps^2\Delta + W(x) - \phi_0(x)
\end{equation}
and $W=1/u$ is defined by \eqref{eqn:W-for-V-V-cut}. Note that $h(x)$ depends on $x$ and is a family of translation invariant operators indexed by $x$. Consequently, by Lemma \ref{lem:L0-explicit-compute}, we see that 
\begin{equation}
	M_{\rm lead} = M_{\rm sc}, \label{eqn:general-MexplicitForm}
\end{equation}
where $M_{\rm sc}$ is given in \eqref{eqn:M-sc-def}.

We now estimate the error terms in \eqref{eqn:M-higher-1} and \eqref{eqn:M-higher-2}. We only consider the term 
\begin{equation}
    \den \oint f_{\rm FD}(\beta(z+V_{\rm cut}-\mu)) R_L' f R(W,\tildeW)
\end{equation}
and the other terms in \eqref{eqn:M-higher-1} and \eqref{eqn:M-higher-2} are similar. By Lemma \ref{lem:den-to-schatten}, for any $2 \leq p,q$ and $\frac{1}{p} + \frac{1}{q}= \frac{1}{2}$,
\begin{align}
    & \left\| \den \oint f_{\rm FD}(\beta(z+V_{\rm cut}-\mu)) R_L'f R(W,\tildeW)\right\|_{L^2(\Om)} \notag \\
    & \hspace{0.5cm}\leq  \oint |f_{\rm FD}(\beta(z+V_{\rm cut}-\mu))| \frac{\eps^{-3/2}}{d(z)} \\
        & \hspace{1cm} \times  \| R_L'(1-\eps^2\Delta)\|_{\mathfrak{S}^p(\Om)} \|(1-\eps^2\Delta)^{-1}f\|_{\mathfrak{S}^q(\Om)} \|R(W,\tildeW)(1-\eps^2\Delta) \|_{\mathfrak{S}^\infty(\Om)} .
\end{align}
Since we chose $V - V_{\rm cut} \gtrsim \delta^{1/4}$ (see \eqref{eqn:cut-size} of Theorem \ref{thm:FD-leading-order}), $d(z)$ can be chosen to be of order $\delta^{1/4} \gg \eps$ (see Figure \ref{fig:cauchy-int-contour-general} with $v=V-V_{\rm cut}$ and $\varphi = \phi_0$). By Kato-Seiler-Simon inequality, Lemma \ref{lem:submain-lemma}, and our choice of scaling in Assumption \ref{ass:pname}, it follows that
\begin{align}
    & \left\| \den \oint f_{\rm FD}(\beta(z+V_{\rm cut}-\mu)) R_L' f R(W,\tildeW)\right\|_{L^2(\Om)} \notag \\ 
    &\hspace{0.5cm} \lesssim   \eps^{-3+1/p} \beta^{-1} e^{-\beta(V_{\rm cut} -\mu - \delta^{-1/4})} \| f\|_{L^q(\Om)},
\end{align}
where the extra factor $\beta$ comes from integrating $f_{\rm FD}(\beta(z-V_{\rm cut}-\mu))$ in $z$. Corollary \ref{cor:kappa-0-fixed-scale} shows that \eqref{eqn:M-sc-explicit-form} holds Lemma \ref{lem:mass-sc} is proved.
\end{proof}

Since $M_{\rm REHF}$ is self-adjoint, we have the following unsurprising corollary for the adjoint $M^*_{\rm sc}$ of $M_{\rm sc}$.
\begin{corollary} \label{cor:mass-sc-star}
Let the assumptions of Theorem \ref{thm:eff-eqn} hold. If $f \in H^2(\Om)$, then
\begin{align}
     &\left\| M_{\rm REHF}f - M_{\rm sc}^* f \right\|_{L^2(\Om)} \lesssim \epsilon^{1/2}  m_0^{-C} \|f\|_{H^2(\Om)}
\end{align}
for some constant $C$ and $m_0$ is given in \eqref{eqn:m-0-def}. 
\end{corollary}
\begin{proof}
We will use the notations $R(W,\tildeW), R_R(\tildeW),R_L(\tildeW)$, and $R_{L/R}'$ given in \eqref{eqn:R-def}, \eqref{eqn:RR-def}, \eqref{eqn:RL-def}, and \eqref{eqn:R-prime-RL-def}, respectively. Let $f \in H^1(\Om)$. Instead of expanding $M_{\rm REHF} f$ as in \eqref{eqn:M-sc-primitive-def}, we switch the roll of $R_L$ and $R_R$ :
\begin{align}
    M_{\rm REHF} f = -\den\oint & f_{\rm FD}(\beta(z+V_{\rm cut}-\mu)) R_R(\tildeW)fR_L(\tildeW) \label{eqn:M-sc-primitive-def-star} \\
    &- \den\oint f_{\rm FD}(\beta(z+V_{\rm cut}-\mu))  [R'_R f L(W, \tildeW) \label{eqn:M-higher-1-star} \\
    &\hspace{1cm} + R(W, \tildeW)f R_L' + R'_R) f R'_L)]. \label{eqn:M-higher-2-star}
\end{align}
Since $M_{\rm sc}$ is computed from \eqref{eqn:M-sc-primitive-def}, we note that \eqref{eqn:M-sc-primitive-def-star} is nothing but $M_{\rm sc}^*$. The higher order terms \eqref{eqn:M-higher-1-star} and \eqref{eqn:M-higher-2-star} are dealt with in the same fashion as Lemma \ref{lem:mass-sc}. The proof of the corollary is complete.
\end{proof}

Let us denote
\begin{equation}
    G(x) = x\log\left(\left|\frac{x+1 }{x - 1} \right| \right).
\end{equation}
In this notation,
\begin{equation}
    M_{\rm sc} = \frac{1}{8\pi^2 \eps^{3}} \int_0^\infty f_{\rm FD} (\beta(t+W-\phi_0 + V_{\rm cut}-\mu)) \frac{1}{\sqrt{4t}} G\left( \frac{\sqrt{4t}}{\eps|\nabla|} \right) dt.
\end{equation}
Define
\begin{equation}
    M_0 = \frac{1}{8\pi^2\eps^{3}}   e^{-\beta(W-\phi_0 + V_{\rm cut}-\mu)}\int_0^\infty e^{-\beta t} \frac{1}{\sqrt{4t}} G\left( \frac{\sqrt{4t}}{\eps|\nabla|} \right) dt. 
\end{equation}
Since $f_{\rm FD}(x)$ approaches $e^{-x}$ exponentially fast if $x$ is large, we have the following corollary. 
\begin{corollary} \label{cor:mass-sc-large-T}
Let the assumptions of Theorem \ref{thm:eff-eqn} hold. If $f \in L^2(\Om)$, then
\begin{equation}
     \left\| M_{\rm REHF}f - M_0 f \right\|_{L^2(\Om)}, \left\| Mf - M_0^* f \right\|_{L^2(\Om)} 
     \lesssim \epsilon^{1/2} m_0^{-C}\|f\|_{H^2(\Om)} \label{eqn:M-sc-explicit-form-large-T}
\end{equation}
 for some constant $C$ and $m_0$ is given in \eqref{eqn:m-0-def}.
\end{corollary}

Now we are ready to prove Theorem \ref{thm:L-lower-bound}.

\begin{proof}[Proof of Theorem \ref{thm:L-lower-bound}]

Let
\begin{equation}
    m =  M_{\rm REHF} - M_0 \label{eqn:m-prime}.
\end{equation}
Lemma \ref{lem:mass-sc} and Corollary \ref{cor:mass-sc-star} show that
\begin{equation}
    \|m f \|_{L^2(\Om)}, \|m^*f \|_{L^2(\Om)} \lesssim \eps^{1/2}  m_0^{-C}  \|f\|_{L^2(\Om)} \label{eqn:m-prime-est}
\end{equation}
for $f \in L^2(\Om)$ where $C$ is a fixed constant. Since $M_{\rm REHF}$ is self-adjoint, by \eqref{eqn:m-prime}, we see that
\begin{align}
     M_{\rm REHF}^2 =&  (M_0^* + m^*)(M_0 + m) \\
        =& M_0^*M_0 + m^* M_0 + M_0^* m + m^* m \label{eqn:mm-bound-1}.
\end{align}
It follows that
\begin{align}
    M_{\rm REHF}^2 =& \frac{1}{64\pi^4\eps^{6}} \left( \int_0^\infty e^{-\beta t} \frac{1}{\sqrt{4t}} G\left( \frac{\sqrt{4t}}{\eps|\nabla|} \right) \right) e^{-2\beta(W-\phi_0 + V_{\rm cut}-\mu)} \left( \int_0^\infty e^{-\beta t} \frac{1}{\sqrt{4t}} G\left( \frac{\sqrt{4t}}{\eps|\nabla|} \right) \right) \notag \\
        &+ m^* M_0 + M_0^* m + m^*m \notag\\
    \gtrsim & \eps^{-6} e^{-2\beta(V_{\max} + \|\phi_0\|_{ \infty}-\mu)} \left( \int_0^\infty e^{-\beta t} \frac{1}{\sqrt{4t}} G\left( \frac{\sqrt{4t}}{\eps|\nabla|} \right) \right)^2 \notag \\
    &+ m^* M_0 + M_0^* m. \notag
\end{align}
By Corollary \ref{cor:kappa-0-fixed-scale}, Lemma \ref{lem:M-basic-form}, and equation \eqref{eqn:m-prime-est}, we see that
\begin{align}
    M_{\rm REHF}^2 \gtrsim&   m_0^{C}  \left( \int_0^\infty e^{-\beta t}\frac{1}{\sqrt{t}} G(\sqrt{t}/\eps \pi |\nabla|) \right)^2 \notag \\
    &- \eps   m_0^{-C} (1-\Delta)^2 \label{eqn:m-2-RHS}
\end{align}
 where $m_0$ is given in \eqref{eqn:m-0-def} and $C$ is a constant.  Since we know that $M_{\rm REHF} \geq 0$, we apply $\max(\cdot, 0)$ to the right hand side of \eqref{eqn:m-2-RHS} (so that we can take its square root). Since the right hand side of \eqref{eqn:m-2-RHS} is purely a function of $|\nabla|$, the usual commutative algebra rules apply. Moreover, since the square-root operator is operator monotone, we conclude that
\begin{align}
    M_{\rm REHF} \gtrsim& m_0^{C}  \int_0^\infty e^{-\beta t} \frac{1}{\sqrt{4t}} G\left( \frac{\sqrt{4t}}{\eps|\nabla|} \right) \notag \\
    &- \eps^{1/2}m_0^{-C} (1-\Delta).  \label{eqn:M-lower-pre-final}
\end{align}
To complete the proof of Theorem \ref{thm:L-lower-bound}, we need the following Lemma, whose proof is delayed until after the current proof.

\begin{lemma} \label{lem:const-m-lower}
\begin{equation}
	\int_0^\infty e^{-t\beta} \frac{1}{\sqrt{4t}} G\left( \frac{\sqrt{4t}}{\eps|\nabla|} \right) dt \geq  \frac{1}{2\sqrt{\beta}(1-\beta \eps^2 \Delta)}. \label{eqn:G-int-est}
\end{equation}
\end{lemma}

Combining equation \eqref{eqn:M-lower-pre-final} and Lemma \ref{lem:const-m-lower}, 
we see that
\begin{align}
 -\Delta + M_{\rm REHF} &\gtrsim -\Delta + m_0^C \frac{1}{\sqrt{\beta}(1-\beta \eps^2 \Delta)} - \epsilon^{1/2}m_0^{-C} (1-\Delta)\\
    &\gtrsim -\Delta + m_0^C  - \eps^{1/2}m_0^{-C}(1-\Delta).
\end{align} 
 By Assumption \ref{ass:pname},  the proof of Theorem \ref{thm:L-lower-bound} is complete, modulo the proof of Lemma \ref{lem:const-m-lower}, which we shall provide in the immediate paragraph that follows.
\end{proof}

\begin{proof}[Proof of Lemma \ref{lem:const-m-lower}]
Let $p = \eps |\nabla|$ and $x = \frac{\sqrt{4t}}{p}$. By our convention, if $p=0^+$, then $G(\infty) = 2$. In this case, \eqref{eqn:G-int-est} is bounded below by 
\begin{equation}
    \int_0^\infty e^{-t\beta} \frac{2}{\sqrt{4t}} dt = \sqrt{\frac{\pi}{\beta}}.
\end{equation}
Otherwise, we assume that $p > 0$. 

By elementary calculus, we have the following two estimates
\begin{align}
	\log\left| \frac{x+1}{x-1} \right| > 2x\  \text{ for }\ x < 1 ,\\
	\log\left| \frac{x+1}{x-1} \right| > \frac{2}{x} \  \text{ for }\  x > 1. 
\end{align}
Equation \eqref{eqn:G-int-est} becomes
\begin{align}
	\int_0^\infty e^{-t\beta} \frac{1}{\sqrt{4t}} G(x) dt \geq & 2\int_0^{p^2} \frac{\sqrt{t}}{p^2} e^{-\beta t} + \int_{p^2}^\infty \frac{1}{\sqrt{t}} e^{-\beta t} \\
 	=& \frac{4}{\beta^{3/2} p^2} \int_0^{\beta p^2} \sqrt{t}e^{-t} + \frac{1}{\sqrt{\beta}} \int_{\beta p^2}^\infty \frac{1}{\sqrt{t}} e^{-t} \\
 	=& \frac{1}{\sqrt{\beta}} \left( \frac{4}{\beta p^2} \int_0^{\beta p^2} \sqrt{t}e^{-t} + \int_{\beta p^2}^\infty \frac{1}{\sqrt{t}} e^{-t} \right). \label{eqn:bracket-term}
\end{align}
We can compute the bracketted term explicitly as a function of $\beta p^2$. Elementary calculus shows that the terms in the bracket in \eqref{eqn:bracket-term} is bounded below by $\frac{1}{2(1+\beta p^2)}$.
\begin{equation}
	\int_0^\infty e^{-t\beta} \frac{1}{\sqrt{t}} G(x) dt \geq   \frac{1}{2\sqrt{\beta}(1+\beta p^2)}.
\end{equation}
This proves \eqref{eqn:G-int-est}.
\end{proof}

\subsection{Proof of Theorem \ref{thm:L-lower-bound-tot}: PL case}

Let $V$ and $\phi = \phi_0$ satisfy the assumption of Theorem \ref{thm:FD-leading-order}. Let 
\begin{equation}
     h  = -\eps^2\Delta + V-\phi_0 -V_{\rm cut}, \label{eqn:h-0-def-PL-case}
\end{equation}
where $V_{\rm cut}$ is given under Theorem \ref{thm:FD-leading-order}. Let also
\begin{equation}
  h_{0}  = u_0^{-1} h u_0 .
\end{equation}
Let $u_0$ denote the Landscape function solving $h u_0 = 1$ and $W_0 = 1/u_0$ its Landscape potential. Define the function, $m_{\rm PL}$, on $\Om$ via
\begin{equation}
    m_{\rm PL} := -\beta  (2\pi\epsilon)^{-3 } \int_{\R^3} dp f'_{\rm FD}(\beta(p^2 + W_0 + V_{\rm cut} - \mu))W_0 \label{eqn:m-PL-def}.
\end{equation}

\begin{theorem} \label{thm:L-lower-bound-PL}
Let the assumptions of Theorem \ref{thm:eff-eqn} hold. The linear operator $L$ given in \eqref{eqn:linear-operator} with $F=F_{\rm PL}$ (see \eqref{eqn:F-LSC-def}) is
\begin{equation}
    L = L_{\rm PL} := -\Delta + m_{\rm PL} h_{0}^{-1} \label{eqn:L-PL-form},
\end{equation}
where $m_{\rm PL}, W_0$, and $u_0$ are seen as multiplication operators. Moreover,
\begin{equation}
    \|L_{\rm PL} f\|_{L^2(\Omega)} \gtrsim \|(-\Delta + m_0^C)f\|_{L^2(\Omega)} \label{eqn:L-lower-bound-PL}
\end{equation}
for some constant $C$ where $m_0$ is given in \eqref{eqn:m-0-def}.
\end{theorem}
\begin{proof}

We differentiate $F_{\rm PL}$ from \eqref{eqn:F-PL-def} to obtain
\begin{align}
    &\hspace{-1cm} d_\phi F_{\rm PL}\mid_{\phi = \phi_0} \varphi \\
    =& -\beta \epsilon^{-3} (2\pi)^{-3/2} \int_{\R^3} dp f'_{\rm FD}(\beta(p^2 + W_0 + V_{\rm cut} - \mu)) W_0^2 d_\phi [(h - \phi)^{-1} 1] \mid_{\phi = 0} \varphi  \\
    =& m_{\rm PL}u_0^{-1} h^{-1} \varphi h^{-1} 1  \\
    =& m_{\rm PL} h_{0}^{-1} \varphi. \label{eqn:M-PL-explicit-form}
\end{align}
This proves \eqref{eqn:L-PL-form}.

Let $M_{\rm PL} = m_{\rm PL}h_{0}^{-1}$. We note that $L=L_{\rm PL}:= -\Delta+M_{\rm PL}$. Since $M_{\rm PL}$ is not self-adjoint, $L_{\rm PL}$ is not self-adjoint. However, $L_{\rm PL}$ is almost self-adjoint as described below. We rewrite
\begin{align}
    L_{\rm PL} =& \left(-\Delta + \frac{1}{2}(M_{\rm PL}+M_{\rm PL}^*)\right) + \frac{1}{2}(M_{\rm PL}-M_{\rm PL}^*) \\
    =&: L_1 + M_{\rm PL}',
\end{align}
where $L_1$ is self-adjoint. We show that $M_{\rm PL}'$ is small. Indeed, we compute
\begin{align}
    2M_{\rm PL}' =& [m_{\rm PL}, h_0^{-1}] \\
    =& -h_0^{-1}[m_{\rm PL}, h_{0}] h_{0}^{-1}  \\
    =& -h_{0}^{-1}u_0^{-1} [m_{\rm PL}, -\eps^2 \Delta] u_0 h_{0}^{-1} \\
    =& -h_{0}^{-1}u_0^{-1} \Big(-\eps^2 \Delta m_{\rm PL} - 2(\eps \nabla m_{\rm PL}) \cdot (\eps \nabla) \Big) u_0 h_{0}^{-1}.
\end{align}
Recalling that $h u_0=1$ where $h$ is given in \eqref{eqn:h-0-def-PL-case}, $\|h^{-1}\|_\infty=\|u_0\|_\infty$ is bounded above by
\begin{equation}
    \frac{1}{\inf V -\|\phi_0\|_{\infty}-V_{\rm cut}} = O(\delta^{-1/4}),
\end{equation}
and
$\|u^{-1}_0\|_\infty$ is bounded above by
\begin{equation}
    \sup V +\|\phi_0\|_{\infty} -V_{\rm cut} = O(\delta^{1/4}).
\end{equation}
 Thus, we see that $\|h_{0}^{-1}u_0^{-1}\|_\infty\lesssim 1$ and $\|u_0h_{0}^{-1}\|_\infty\lesssim \delta^{-1/2}$. Hence,  
\begin{align}
    \|M_{\rm PL}'f\|_{ 2 } \lesssim& \,    \|(\eps^2 \Delta m_{\rm PL} +2(\eps \nabla m_{\rm PL}(\eps \nabla)) u_0 h_{0}^{-1}f \|_{ 2 } \notag\\
    \lesssim&\,    \delta^{-1/2}\|(-\eps^2 \Delta m_{\rm PL}) \|_{ 4 }\, \| f\|_{ 4 }   +  \|(\eps \nabla m_{\rm PL}) \|_{ 4 } \|(\eps \nabla)u_0 h_{0}^{-1}f \|_{ 4 }.
\end{align}
It follows by the Sobolev inequality, Theorem \ref{thm:useful-L-p-est-on-DW}, Corollary \ref{cor:kappa-0-fixed-scale}, and definition \eqref{eqn:m-0-def} that 
\begin{equation}
    \|M_{\rm PL}'f\|_{ 2 } \lesssim m_0^{-C} \eps^{1/4}\|f\|_{H^1(\Om)}. \label{eqn:M'-corrected-upper-bound}
\end{equation}

Next, we provide a lower bound for $\frac{1}{4}(M_{\rm PL}+M_{\rm PL}^*)^2$. We compute
\begin{align}
    \left( \frac{M_{\rm PL}+M_{\rm PL}^*}{2} \right)^2 =& (M_{\rm PL}^* + M_{\rm PL}')(M_{\rm PL} - M_{\rm PL}') \\
    =& M_{\rm PL}^*M_{\rm PL} + M_{\rm PL}'M_{\rm PL} - M_{\rm PL}^* M_{\rm PL}' - (M_{\rm PL}')^2. \label{eqn:M-plus-M-star-expand}
\end{align}
 We denote $v_{\max} = \sup V +\|\phi_0\|_{\infty} -V_{\rm cut} = O(\delta^{1/4})$. Using the explicit form of $M_{\rm PL}$ in \eqref{eqn:M-PL-explicit-form}, we see that
\begin{align}
    M_{\rm PL}^*M_{\rm PL} =& \,  u_0 h^{-1} u_0^{-1} m_{\rm PL}^2 u_0^{-1} h^{-1} u_0 \\
    \gtrsim & \,  \beta^2 \eps^{-6} e^{-2\beta(V_{\rm cut} - \mu + C\delta^{1/4})} {u_0(v_{\max}-\eps^2\, \Delta)^{-2}u_0} \\
    \gtrsim &\,  m_0^{2C} u_0(v_{\max}-\eps^2\, \Delta)^{-2}u_0, \label{eqn:u-0-double}
\end{align}
where $m_0$ is given in \eqref{eqn:m-0-def}, $C$ is a constant. Note that the last line follows from the scaling Assumptions \ref{ass:semi-classical} -- \ref{ass:temperature} and Corollary \ref{cor:kappa-0-fixed-scale}.  {To estimate \eqref{eqn:u-0-double}, we have that
\begin{equation}
    (v_{\max}-\eps^2\, \Delta)^{-1}u_0 = u_0(v_{\max}-\eps^2\, \Delta)^{-1} + (v_{\max}-\eps^2\, \Delta)^{-1}[u_0, -\eps^2\Delta] (v_{\max}-\eps^2\, \Delta)^{-1}.
\end{equation}
It follows from \eqref{eqn:u-0-double}, Theorem \ref{thm:useful-L-p-est-on-DW}, and Sobolev's inequality (which depends on $|\Om|$) that for any $f \in L^2(\Om)$,
\begin{align}
    \lan f, M^*Mf \ran \gtrsim & \,  m_0^{2C}\|(v_{\max}-\eps^2\, \Delta)^{-1}u_0 f\|_{2 } \\
    \gtrsim& \,  m_0^{2C}\|u_0(v_{\max}-\eps^2\, \Delta)^{-1} f\|_{ 2 } - m_0^{-C'}\eps^{1/4}\|f\|_{ 2 } \\
    \gtrsim& \,  m_0^{2C}\|(v_{\max}-\eps^2\, \Delta)^{-1} f\|_{ 2 }
\end{align}
for suitable and possibly different constants $C$ and $C'$.} Together with \eqref{eqn:M'-corrected-upper-bound}, \eqref{eqn:L-lower-bound-PL} is proved by \eqref{eqn:M-plus-M-star-expand}.
\end{proof}

\subsection{Proof of Theorem \ref{thm:L-lower-bound-tot}: LSC case}
Similar to the PL case, let $V$ and $\phi = \phi_0$ satisfy the assumption of Theorem \ref{thm:FD-leading-order}. Let $u_0$ denote the Landscape function associated to the Hamiltonian
\begin{equation}
    -\eps^2\Delta + V - V_{\rm cut},
\end{equation}
where $V_{\rm cut}$ is given under Theorem \ref{thm:FD-leading-order}. 
Define the function, $m_{\rm LSC}$, on $\Om$ via
\begin{equation}
    m_{\rm LSC} := -\beta   (2\pi\epsilon)^{-3/2} \int_{\R^3} dp f'_{\rm FD}(\beta(p^2 + W_0 - \phi_0 + V_{\rm cut} - \mu)) \label{eqn:m-LSC-def},
\end{equation}
where $W_0=1/u_0$ is the Landscape potential. 

\begin{theorem} \label{thm:L-lower-bound-LSC}
Let the assumptions of Theorem \ref{thm:eff-eqn} hold. The linear operator $L$ given in \eqref{eqn:linear-operator} with $F= F_{\rm LSC}$ (see \eqref{eqn:F-LSC-def}) is
\begin{equation}
    L = -\Delta+m_{\rm LSC},
\end{equation}
where $m_{\rm LSC}$ is seen as a multiplication operator. Moreover, $L$ is self-adjoint on $L^2(\Om)$ and 
\begin{equation}
    L \gtrsim -\Delta + m_0^C \label{eqn:L-lower-bound-LSC}
\end{equation}
for some constant $C$ where $m_0$ is given in \eqref{eqn:m-0-def}.
\end{theorem}
\begin{proof}
We differentiate $ F_{\rm LSC}$ from \eqref{eqn:F-LSC-def} to obtain
\begin{align}
  d_\phi  F_{\rm LSC}\mid_{\phi = \phi_0} \phi = &  -\beta \epsilon^{-3} (2\pi)^{-3/2} \int_{\R^3} dp f'_{\rm FD}(\beta(p^2 + W_0-\phi_0  + V_{\rm cut} - \mu)) \phi \\
  =& m_{\rm LSC} \phi,   
\end{align}
where $m_{\rm LSC}$ is given in \eqref{eqn:m-LSC-def}. By assumptions \ref{ass:semi-classical} - \ref{ass:temperature}, we may replace $f'_{\rm RD}(x)$ by $e^{-x}$, and we obtain
\begin{equation}
    m_{\rm LSC} \gtrsim  \beta \epsilon^{-3} e^{-\beta( W_0+ \|\phi_0\|_{ \infty} + V_{\rm cut} - \mu)} \int_{\R^3} dp e^{-\beta p^2}  
    \gtrsim  \beta^{-1/2} \eps^{-3} e^{-\beta( V_{\max} + \|\phi_0\|_{ \infty} -\mu)}. 
\end{equation}
By Corollary \ref{cor:kappa-0-fixed-scale}, Assumptions \ref{ass:semi-classical} - \ref{ass:temperature}, we see that $m_{\rm LSC} \gtrsim m_0^C$ for some constant $C$. This proves \eqref{eqn:L-lower-bound-LSC} and completes the proof of Theorem \ref{thm:L-lower-bound-LSC}.
\end{proof}

\section{Nonlinear Analysis} \label{sec:nonlin-anal}
In this section we prove Theorem \ref{thm:nonlin-tot} in three parts: in each of the following subsections, we prove a version of Theorem \ref{thm:nonlin-tot} for the case of REHF, PL, and LSC in Theorems \ref{thm:nonlin}, \ref{thm:nonlin-PL-est}, and \ref{thm:nonlin-LSC-est}, respectively.

\subsection{Proof of Theorem \ref{thm:nonlin-tot}: REHF case}

\begin{theorem} \label{thm:nonlin}
Let the assumptions of Theorem \ref{thm:eff-eqn} hold. The nonlinear operator $N$ (defined in \eqref{eqn:N-def}) of the REHF equation has the following estimate
\begin{equation}
	\|N(\phi_1) - N(\phi_2)\|_{L^2(\Om)} \leq C_1 m_0^{-C_2} (\|\phi_1\|_{H^1(\Om)} + \|\phi_2\|_{H^1(\Om)}) \|\phi_1 - \phi_2\|_{L^2(\Om)} \label{eqn:nonlin}
\end{equation}
for $\phi_1$ and $\phi_2$ in $H^1(\Om)$  provided that $\|\phi_i\|_{H^1(\Om)} \lesssim m_0^{C_3}$ for some constant $C_3$ large enough, where $m_0$ is given in \eqref{eqn:m-0-def}.
\end{theorem}
\begin{proof}[Proof of Theorem \ref{thm:nonlin}]
By \eqref{eqn:N-def}, we see that
 {\begin{equation}
    N(\phi) =- \den f_{\rm FD}(\beta(h - \phi - \mu))+ \den f_{\rm FD}(\beta(h - \mu)) +M\phi ,
\end{equation} }
where 
\begin{equation}
    h = -\eps^2 \Delta + V - \phi_0 - V_{\rm cut} \label{eqn:nonlin-h-def}
\end{equation}
(recall from Theorem \ref{thm:FD-leading-order} and \eqref{eqn:cut-size} our choice of $V_{\rm cut}$).  We remark that since $\phi$ is not necessarily real, the operator $h$ is not self-adjoint in general. However, its spectrum lies within a tubular neighborhood of the real with with width $O(\|\phi\|_\infty) \lesssim O(\|\phi\|_{H^2}) \lesssim m_0^{C_3} = \eps^{C_3\delta^{1/4}}$ for some constant $C_3$, by Assumption of Theorem \ref{thm:nonlin}. In particular, the spectrum of $h-\phi$ does not intersect our contour of integration since the poles of $f_{\rm FD}(\beta(z+V_{\rm cut} - \mu))$ are $-(V_{\rm cut} - \mu)+i\pi\beta^{-1}\Z$ (see Figure \ref{fig:cauchy-int-contour-general}).   Thus, recall that the resolvent identity is
\begin{equation}
	(z-A)^{-1} - (z-B)^{-1} = (z-A)^{-1}(A-B)(z-B)^{-1}. \label{eqn:resolventID}
\end{equation}
Using the Cauchy-integral, the resolvent identity, and \eqref{eqn:M-def}, we arrive at an explicit formula for for $N$:
\begin{equation}
	N(\phi) := -\den \oint f_{\rm FD}(\beta(z+V_{\rm cut}-\mu) \,  (z-(h-\phi))^{-1} (\phi (z-h)^{-1})^2 \label{eqn:nonlinExplicit},
\end{equation}
 where $\oint$ is given in \eqref{eqn:oint-def}.

Applying the resolvent identity to \eqref{eqn:nonlinExplicit} iteratively with $A = h- \phi$ and $B = h$ ($h$ is defined in \eqref{eqn:nonlin-h-def}), we arrive at
\begin{equation}
	N(\phi) = \sum_{n \geq 2} \oint f_{\rm FD}(\beta(z+V_{\rm cut}-\mu)) (-1)^n \den (z -h)^{-1} (\phi (z-h)^{-1})^n \label{eqn:NexpandinNn},
\end{equation}
whenever the series converges, which we will demonstrate. Let
\begin{equation}
    N_n(\phi) = \oint f_{\rm FD}(\beta(z+V_{\rm cut}-\mu)) (-1)^n \den (z -h)^{-1} (\phi (z-h)^{-1})^n
\end{equation}
denote the $n$-th order nonlinearity. Our goal is to estimate the difference in the individual $n$-th order nonlinearities 
\begin{multline}
  N_n(\phi_1) - N_n(\phi_2)   
	=  \oint f_{\rm FD}(\beta(z+V_{\rm cut}-\mu))  \\
	  \times (-1)^n\den (z-h)^{-1} (\phi_1 (z-h)^{-1})^n - (z-h)^{-1} (\phi_2 (z-h)^{-1})^n. \label{eqn:NnDef}   
\end{multline}
To do so, we use the following expansion of $n$-th degree monomials:
\begin{equation}
    a^n - b^n = (a-b)a^{n-1} + b(a-b)a^{n-2} + \cdots + b^{n-1}(a-b).
\end{equation}
Using this pattern, we see that
\begin{multline}
  N_n(\phi_1) - N_n(\phi_2)  \\
	= \oint f_{\rm FD}(\beta(z+V_{\rm cut}-\mu)) \den (-1)^n(z-h)^{-1} (\phi_1 - \phi_2)(z-h)^{-1}(\phi_\# (z-h)^{-1})^{n-1} \\
	+ \text{$n-1$ similar terms},  
\end{multline}
where $\phi_\#$ denotes $\phi_1$ or $\phi_2$. Using the standard Schatten $p$-norm $\mathfrak{S}^p(\Om)$ and the Kato-Seiler-Simon inequality, (see Appendix \ref{sec:per-vol-set-up}) and by Lemma \ref{lem:den-to-schatten} of the Appendix, we see that
\begin{align}
	& \hspace{-1cm}\| N_n(\phi_1) - N_n(\phi_2)\|_{ 2 } \notag \\
	\leq& n \oint \eps^{-3/2}  |f_{\rm FD}(\beta(z+V_{\rm cut}-\mu))| \|(\phi_1 - \phi_2)(z-h)^{-1}(\phi_\# (z-h)^{-1})^{n-1} \|_{\mathfrak{S}^2(\Om)}  \notag\\
	\lesssim&  n\oint \eps^{-3} |e^{-\beta(z+V_{\rm cut}-\mu)}| \|(\phi_\# (z-h)^{-1})^{n-1}\|_{\mathfrak{S}^\infty(\Om)} \| \phi_1 - \phi_2\|_{ 2 }.
\end{align}
It follows that
\begin{align}
     \| N_n(\phi_1) - N_n(\phi_2)\|_{L^2(\Om)}& \notag \\
	\lesssim&\,  n m_0^{-C} \|(\phi_\# (z-h)^{-1})^{n-1}\|_{\mathfrak{S}^\infty(\Om)}
	  \|\phi_1 - \phi_2\|_{ 2 },
\end{align}
by the scaling in Assumptions \ref{ass:semi-classical} - \ref{ass:temperature} and Corollary \eqref{cor:kappa-0-fixed-scale}. Denote $d(z)$ to be the distance from $z$ to the contour. Recall that $m_0$ is given in \eqref{eqn:m-0-def} and $\inf_z d(z) = O(V_{\min} - V_{\rm cut}) = O(\delta^{1/4})$ by the choice of the contour (see Theorem \ref{thm:FD-leading-order}). 
We see that by H\"{o}lder's and Sobolev inequalities,
\begin{align}
    \|(z-h)^{-1}(\phi_\# (z-h)^{-1})^{n-1}\|_{\mathfrak{S}^\infty(\Om)} \leq& \,  \|(z-h)^{-1}\|_{\mathfrak{S}^\infty(\Om)} \|\phi_\# (z-h)^{-1} \|_{\mathfrak{S}^\infty(\Om)}^{n-1} \\
    \lesssim& \,  nm_0^{-C}\delta^{-n/4}\|\phi_\#\|_{ 4 }^{n-1} \\
    \lesssim& \,  nm_0^{-C}\delta^{-n/4} \|\phi_\#\|_{H^1(\Om)}^{n-1} .
\end{align}
Combining with \eqref{eqn:NexpandinNn} and the assumption that $\|\phi_i\|_{H^2(\Om)} \lesssim m_0^{C_3}$ (for some constant $C_3$ large) is sufficiently small, we conclude that the claim \eqref{eqn:nonlin} is proved.
\end{proof}

\subsection{Proof of Theorem \ref{thm:nonlin-tot}: LSC case}
We will prove the simpler LSC case first before embarking on the tedious  yet similar proof of the PL case. 

\begin{theorem} \label{thm:nonlin-LSC-est}
Let the assumptions of Theorem \ref{thm:eff-eqn} hold. Let $\phi_1, \phi_2 \in H^2(\Om)$ with $\|\phi_i\|_{H^2(\Om)} \lesssim m_0^{C_3}$ for some large constant $C_3$ ($m_0$ is defined in \eqref{eqn:m-0-def}). The nonlinear operator $N$ implicitly defined in \eqref{eqn:N-def} with $F= F_{\rm LSC}$ (see \eqref{eqn:F-LSC-def}) satisfies the following estimates:
\begin{equation}
	\|N(\phi_1) - N(\phi_2)\|_{L^2(\Om)} \leq C_1  m_0^{-C_2} \big(\|\phi_1\|_{H^2(\Om)} + \|\phi_2\|_{H^2(\Om)}\big) \|\phi_1 - \phi_2\|_{H^2(\Om)}, \label{eqn:nonlin-LSC}
\end{equation}
where $C_1$ and $C_2$ are constants.
\end{theorem}
\begin{proof}
Let $    h =-\varepsilon^2\, \Delta+V-V_{\rm cut}.
$
Let $u_0$ denote the Landscape function solving $h u_0=1$ and $W=1/u_0$ its Landscape potential. 
 Similar to the REHF equation, explicitly, by \eqref{eqn:F-LSC-def} and Theorem \ref{thm:L-lower-bound-LSC}, we see that
\begin{multline}
   N(\phi) = \frac{1}{(2\pi\epsilon)^{3 }} \int_{\R^3} dp   \Big(f_{\rm FD}(\beta(p^2 + W - \phi_0  + V_{\rm cut} - \mu))    \\
    - f_{\rm FD}(\beta(p^2 + W - \phi_0 - \phi+ V_{\rm cut} - \mu))\\ 
     + \beta f'_{\rm FD}(\beta(p^2 + W - \phi_0 + V_{\rm cut} - \mu)) \phi \Big).   
\end{multline}

By Assumptions \ref{ass:semi-classical} - \ref{ass:temperature}, we may replace $f'_{\rm FD}, f_{\rm FD}(x)$ by $e^{-x}$. It follows that
\begin{align}
    &\hspace{-2cm}|N(\phi_1)-N(\phi_2)| \notag \\ \lesssim& \,  \epsilon^{-3} e^{ {-\beta(W-\|\phi_0\|_\infty+V_{\rm cut} - \mu)}}  \int_{\R^3} dp e^{-\beta p^2} \Big|e^{-\beta \phi_1}-e^{-\beta \phi_2}  +\beta (\phi_1-\phi_2)\Big|\notag \\
    \lesssim& \,  \beta^{-3/2}\eps^{-3} e^{-\beta(V_{\min} - \|\phi_0\|_{ \infty} - \mu)} \,  |e^{-\beta \phi_1}-e^{-\beta \phi_2}  +\beta (\phi_1-\phi_2)| \notag \\
    \lesssim& \, 
    m_0^{-C} \,  |e^{-\beta \phi_1}-e^{-\beta \phi_2}  +\beta (\phi_1-\phi_2)| \label{eq:382}.
\end{align}
Since $\|\phi_i\|_{ \infty} \lesssim \|\phi_i\|_{H^2(\Om)} \lesssim m_0^{C_3}$ is smaller than $1/\beta$ by Assumption \ref{ass:temperature}, the Taylor expansion of $e^{x}$ and \eqref{eq:382} proves equation \eqref{eqn:nonlin-LSC}.

\end{proof}

\subsection{Proof of Theorem \ref{thm:nonlin-tot}: PL case}

\begin{theorem} \label{thm:nonlin-PL-est}
Let the assumptions of Theorem \ref{thm:eff-eqn} hold. Assume that $\phi_1, \phi_2 \in H^2(\Om)$ and $\|\phi_i\|_{H^2(\Om)} \lesssim m_0^{C_3}$ for some large constant $C_3$ ($m_0$ is defined in \eqref{eqn:m-0-def}). The nonlinear operator $N$ implicitly defined in \eqref{eqn:N-def} with $F=F_{\rm PL}$ (see \eqref{eqn:F-PL-def}) satisfies the following estimates:
\begin{equation}
	\|N(\phi_1) - N(\phi_2)\|_{L^2(\Om)} \leq C_1  m_0^{-C_2} (\|\phi_1\|_{H^2(\Om)} + \|\phi_2\|_{H^2(\Om)}) \|\phi_1 - \phi_2\|_{H^2(\Om)}, \label{eqn:nonlin-PL}
\end{equation}
where $C_1$ and $C_2$ are constants.
\end{theorem}
\begin{proof}
Let 
\begin{equation}
    h =-\varepsilon^2\, \Delta+V-\phi_0-V_{\rm cut}.
\end{equation}
Let $u_0$ denote the Landscape function solving $h u_0=1$ and $W_0=1/u_0$ its Landscape potential. Let $W(\phi)= [(h-\phi)^{-1}1]^{-1}$ be the Landscape potential of $h-\phi$. 
 Similar to the LSC case, by \eqref{eqn:F-PL-def} and Theorem \ref{thm:L-lower-bound-PL}, we see that
\begin{multline}
    N(\phi) = \frac{1}{(2\pi\epsilon)^{3 }} \int_{\R^3} dp \Big(f_{\rm FD}(\beta(p^2 + W_0 + V_{\rm cut} - \mu))- f_{\rm FD}(\beta(p^2 + W(\phi) + V_{\rm cut} - \mu))   \\
     + \beta   f'_{\rm FD}(\beta(p^2 + W_0 + V_{\rm cut} - \mu)) d_\phi W \mid_{\phi = 0} \phi\Big).  
\end{multline}
Again we  replace $f'_{\rm FD}, f_{\rm FD}(x)$ by $e^{-x}$ . Denote by $W_1=W(\phi_1)$ and $W_2=W(\phi_2)$.  Similar to \eqref{eq:382}, it follows that
\begin{equation}
     |N(\phi_1)-N(\phi_2)|  \lesssim   m_0^{ - C}   \Big|e^{-\beta (W_1-W_0)} - e^{-\beta (W_2-W_0)} -\beta   d_\phi W \mid_{\phi = 0} (\phi_1-\phi_2) \Big|  .
\end{equation}
Direct computation shows that 
\begin{equation}
    W(\phi)=\frac{1}{(h-\phi)^{-1}1}=\frac{1}{u_0\big(1+u_0^{-1}\Sigma_{n\ge 1}(-h^{-1}\phi)^n\, u_0\big)} \label{eqn:W-phi-expansion}
\end{equation}
and 
\begin{equation}
    d_\phi W_{\phi=0} \phi' = W_0^2 h^{-1} \phi' h^{-1} 1 . 
\end{equation}
Hence,
\[
W(\phi)-W_0-d_\phi W \mid_{\phi = 0}\phi=W_0^2(h^{-1}\phi)^2u_0+{\textrm{higher order terms of}}\ h^{-1}\phi. 
\]
As before, let $v_{\rm min} :=V_{\min} - V_{\rm cut} = O(\delta^{1/4})$ (see \eqref{eqn:cut-size}). Since $\|h^{-1}\phi\|_\infty\le v_{\min}^{-1}\, \delta \ll 1$ and $\delta \ll v_{\min}^C$ and  similar to the estimate of the nonlinearity in the LSC equation, one has 
\begin{equation}
    \big\|W(\phi)-W_0-d_\phi W \mid_{\phi = 0}\phi \big\|_2
\lesssim m_0^{-C}\|\phi^{2}\|_{H^2}
\end{equation}
for some constant $C > 0$. Since $\|\phi_i\|_{ \infty} \lesssim \|\phi_i\|_{H^2(\Om)}$.
Therefore, 
\begin{align}
 &\Big\| e^{-\beta (W_1-W_0)} - e^{-\beta (W_2-W_0)} - \beta     d_\phi W \mid_{\phi = 0} (\phi_1-\phi_2)  \Big\|_2 \notag \\
 &\hspace{6cm} \lesssim \beta^2 m_0^{-C} (\|\phi_1\|_{H^2}+\|\phi_2\|_{H^2})\|\phi_1-\phi_2\|_{H^2}. \label{eq:tmp0} 
\end{align}
 Assumption \ref{ass:pname} and \ref{ass:temperature} allows the $\beta^2$ to be absorbed into $m_0^{-C}$. This implies equation \eqref{eqn:nonlin-PL}.
\end{proof}

\appendix
\section{Trace per volume and associated Schatten norm estimates} \label{sec:per-vol-set-up}

Let $\lat$ denote a Bravais lattice in $\R^3$ and $\Om$ its fundamental domain (for example, the Wigner-Seitz cell). For $l \in \R^3$, let $U_l$ denote the translation operator
\begin{equation}
    (U_l f)(x) = f(x-l).
\end{equation}
An operator on $L^2(\R^3)$ is said to be translation invariant if $A$ commutes with $U_l$ for all $l \in \R^3$. It is said to be ($\lat$) periodic if $A$ commutes with $U_l$ for all $l \in \lat$. 

Let $\Tr$ denote the usual trace on $L^2(\R^3)$. It is evident that no periodic or translation invariant operators have a finite trace. Nonetheless, we consider trace per volume $\Om$ defined via
\begin{equation}
    \Tr_{\Om} A := \frac{1}{|\Om|} \Tr \chi_\Om A , \label{eqn:trace-per-vol}
\end{equation}
where $\chi_\Om$ is the indicator function of $\Om$. If $A$ is periodic, $\Tr_\Om$ is independent of translates of $\Om$. Associated to $\Tr_\Om$ is a family of periodic Schatten spaces $\mathfrak{S}^p(\Om)$ given via
\begin{equation}
    \mathfrak{S}^p(\Om) = \{ A \in \mathcal{B}(L^2(\R^3)) \text{ and ($\lat$) periodic} : \|A\|_{\mathfrak{S}^p(\Om)} < \infty \}
\end{equation}
where
\begin{equation}
    \|A\|_{\mathfrak{S}^p(\Om)}^p = \Tr_\Om (A^*A)^{p/2}.
\end{equation}
The $\mathfrak{S}^p(\Om)$ norm inherits most inequality estimates from the usual Schatten $p$-norm with the notable exception that
\begin{equation}
    \|A\|_{\mathfrak{S}^\infty(\Om)} \leq C\|A\|_{\mathfrak{S}^p(\Om)}
\end{equation}
{\it fails} to hold for $1 \leq p < \infty$ and for any $C > 0$.

Given an operator $A$, its density $\den A$ is a measurable function on $\Om$, if it exists, defined via the Riesz representation theorem and the formula
\begin{equation}
    \Tr f A = \int_{\R^3} f\den A \label{eqn:den-def-via-riesz}
\end{equation}
for any $f \in C^\infty_c(\R^3)$, where $f$ on the left hand side is regarded as a multiplication operator on $L^2(\R^3)$. When $A$ has an integral kernel $A(x,y)$, that is
\begin{equation}
    (Af)(x) = \int_{\R^3} A(x,y)f(y) dy,
\end{equation}
then,
\begin{equation}
    (\den A)(x) = A(x,x),
\end{equation}
whenever $A(x,x)$ is defined and unambiguous. We outline a few special cases where $\den A$ is defined and give its estimates below, which will be frequently used in the proof of the main results. 

\begin{lemma} \label{lem:den-to-schatten}
Let $d = 3$ and $\frac{3}{2} < p < 3$. Suppose that $A \in \mathfrak{S}^p(\Om)$ and $R$ be given by \eqref{eq:R00}, 
then $\den  AR$ and $\den  RA \in L^2(\Om)$. Moreover,
\begin{equation}
    \|\den  AR \|_{L^p(\Om)}, \|\den R A\|_{L^p(\Om)}  \lesssim \frac{\epsilon^{-3/q}}{d(z)} \|A\|_{\mathfrak{S}^p(\Om)} ,
\end{equation}
where $\frac{1}{q}+ \frac{1}{p}=1$.
\end{lemma}
\begin{proof}
We prove the case for $RA$ only as the case for $AR$ is similar. We use the $L^p(\Om)$-$L^q(\Om)$ duality (where $\frac{1}{p} + \frac{1}{q} = 1$ for $\frac{3}{2} < p, q < \infty$). Let $\phi \in L^p(\Om)$ and apply H\"{o}lder's inequality to 
\begin{equation}
    \Tr_\Om [(\phi R) A] \lesssim \|\phi R\|_{\mathfrak{S}^q(\Om)} \|A\|_{\mathfrak{S}^p(\Om)},
\end{equation}
where $\Tr_\Om$ is the trace per volumne $\Om$. Kato-Seiler-Simon inequality show
\begin{equation}
    \Tr_{\Om} [(\phi R) A] \leq C \frac{\epsilon^{-3/q}}{d(z)} \|\phi\|_{L^q(\Om)}\|A\|_{\mathfrak{S}^p(\Om)},
\end{equation}
where $\Im z$ is the imaginary part of $z$. The proof is now completed by the $L^q(\Om)$-$L^p(\Om)$ duality and the Riesz representation theorem:
\begin{equation}
    |\Om|\Tr_\Om(\phi RA) = \Tr(\phi RA) = \lan \phi, \den(RA) \ran_{L^2(\Om)}.
\end{equation}
\end{proof}

\section{Existence of solution to the Poisson-Landscape equation} \label{app:PL-existence}

We show that the LSC equation (see \eqref{eqn:general-F-def} and \eqref{eqn:F-LSC-def}) has a solution by minimizing its associated energy functional. To this end, let $\eta(p, x)$ be periodic in $x$ and $\rho_\eta = \int_{\R^3} dp \, \eta(p, x)$. Let $s(x) = -\frac{1}{2}(x\log(x) + (1-x)\log(1-x))$. We define the entropy functional
\begin{equation}
    S(\eta) = \int_{\R^3 \times \Om}  dpdx \, s(\eta(p,x)),
\end{equation}
whenever the integral is convergent. Otherwise we set $S(\eta) = \infty$. Finally, we define
\begin{align}
   \mathcal{F}_{\rm LSC}(\eta)  =& \int_{\R^3 \times \Om} dp dx \,  (\eps^2 p^2+W+V_{\rm cut}) \eta(p, x) \notag \\
    &+ \frac{1}{2}\lan (\rho_\eta - \kappa), (-\Delta)^{-1}(\rho_\eta - \kappa) \ran_{L^2(\Om)} \notag\\
    &- \beta^{-1}S(\eta).
\end{align}
The associated space on which we perform our minimization is
\begin{align}
    D_{\kappa} = \{ \eta  \in  L^1(\R^3_p\times \Om_x, &(1+p^2)dpdx) : 0 \leq \eta \leq 1,  \notag \\
    & \int_{\R^3 \times \Om} \eta = \int_{\Om} \kappa, \text{ and } \rho_\eta \in \kappa + \dot{H}^{-1}(\Om) \}, \label{eqn:D-kappa}
\end{align}
where, using the notation $\nabla^{-1} := \nabla(-\Delta)^{-1}$,
\begin{equation}
    \dot{H}^{-1}(\Om) = \left\{ f : \int_\Om f = 0 \text{ and } \nabla^{-1} f \in L^2(\Om) \right\}.
\end{equation}
We note that $D_\kappa$ is not an affine space, but it is convex. Let
\begin{align}
    L_1 \eta &= \int_{\R^3} dp \, \eta(p,x), \\
    L_2 \eta &= \int_\Om dx \, \eta(p,x),
\end{align}
and define projections
\begin{align}
    P =& \frac{1}{|\Om|} L_2 ,\\
    \bar P =& 1-P.
\end{align}
For simplicity, we will denote
\begin{equation}
    L^p(\R^3\times \Om, (1+p)^2dpdx) = L^p((1+p)^2dpdx) 
\end{equation}
interchangeably.

\begin{theorem}[LSC Existence and Uniqueness] \label{thm:PL-existence-uniqueness}
Assume that $\bar P \kappa \in \dot{H}^{-1}$. Then $\mathcal{F}_{\rm LSC}$ is convex and has a unique minimizer $\eta \in D_\kappa$. If $-\Delta \phi = \kappa - \rho_\eta$, then $\phi$ solves the Poisson-Landscape equation \eqref{eqn:general-F-def} with $F=F_{\rm LSC}$ (cf. \eqref{eqn:F-LSC-def}) with $\mu$ induced by the Lagrangian multiplier of the constraint.
\end{theorem}

\begin{proof}[Proof of Theorem \ref{thm:PL-existence-uniqueness}]

We equip $D_\kappa$ with the norm
\begin{equation}
    \|\eta\|_{D_\kappa} = \|\eta(p,x)\|_{L^1((1+p^2)dpdx)} + \| \bar P\rho_\eta  \|_{\dot{H}^{-1}(\Om)}.
\end{equation}
Note that $D_\kappa$ is closed with respect to this norm. 

{\bf Step 1: Euler-Lagrange equation.}
Let $h(p,x) = \eps^2 p^2 + W + V_{\rm cut}$
The energy functional can be rewritten as
\begin{equation}
    \mathcal{F}_{\rm LSC}(\eta) =  L_1L_2 (h\eta) + \frac{1}{2} \|\nabla^{-1} (L_1 \eta - \kappa)\|_{L^2(\Om)}^2 - \beta^{-1} L_1L_2 s(\eta). \label{eqn:E-compact-form}
\end{equation}
From this, the Euler-Lagrange equation subject to $L_1L_2 \eta = L_2 \kappa$ is
\begin{equation}
    h\eta + (-\Delta)^{-1}(\rho_\eta -   \kappa) - \beta^{-1} s'(\eta) -\mu \eta = 0
\end{equation}
for a suitable $\mu$ due to the constraint $L_1L_2 \eta = L^2 \kappa$. Solving for $\eta$, we see that
\begin{align}
    \eta = f_{\rm FD}(\beta(h-(-\Delta)^{-1}(\rho_\eta -   \kappa) - \mu)). \label{eqn:EL-PL-pre}
\end{align}
Integrating equation \eqref{eqn:EL-PL-pre} with respect to $p$, we define
\begin{align}
    \rho_\eta = \int_{\R^3} dp \, f_{\rm FD}(\beta(\eps^2 p^2 + W + V_{\rm cut} - \phi - \mu)).
\end{align}
Finally, set $\phi = (-\Delta)^{-1}(\rho_\eta -   \kappa)$. We see that $(\phi, \mu)$ solves \eqref{eqn:general-F-def} with \eqref{eqn:F-LSC-def}.

{\bf Step 2: Coercivity.} 
Note that without the interaction term (Coulomb term), an unconstrained minimizer to
\begin{align}
    \int_{\R^3 \times \Om} dp dx \,  (\frac{1}{2} \eps^2 p^2+W+V_{\rm cut}) \eta(p, x) \notag - \beta^{-1}S(\eta)
\end{align}
is
\begin{align}
    \eta = f_{\rm FD}(\beta h ).
\end{align}
This shows that
\begin{align}
    \mathcal{F}_{\rm LSC}(\eta) \geq& \frac{1}{2} L_1L_2 (h\eta) + \frac{1}{2} \|\nabla^{-1} (L_1 \eta - \kappa)\|_{L^2(\Om)}^2 - C \\
    \gtrsim& \|\eta\|_{D_\kappa} - C
\end{align}
for $\|\eta\|_{D_\kappa}$ large.

{\bf Step 3: Convergent subsequence.} Since $\mathcal{F}_{\rm LSC}$ is coercive, we can find a minimizing sequence $\eta_n$. Note that since $0 \leq \eta_n \leq 1$, we have that $\eta_n^p \leq \eta_n$ for all $1 \leq p  < \infty$. It follows that
\begin{align}
    \|\eta_n\|_{L^p((1+p^2)dpdx)}^p \leq \|\eta_n\|_{L^1((1+p^2)dpdx)} < \infty.
\end{align}
In particular, $\eta_n$ converges weakly to some $\eta(p)$ for each $1 < p < \infty$ in $L^p(\R^3\times \Om, (1+p^2)dpdx)$. By testing against compactly supported smooth functions (i.e. 
\[
    \lan \varphi, \eta_n \ran_{L^2((1+p^2)dpdx)} \rightarrow \lan \varphi, \eta \ran_{L^2((1+p^2)dpdx)}
\]
for $n \rightarrow \infty$, where $\varphi$ is smooth compactly supported on $\R^3 \times \Om$, which is in $L^p$ for any $1 \leq p \leq \infty$, we see that the $\eta(p) = \eta(q)$ for any $1 < p, q < \infty$.

So we will denote by $\eta$ the common limit. Moreover, note that for any measurable $f$ on $\R^3 \times \Om$, 
\begin{align}
    \int_{\R^3 \times \Om}  \eta_n f  =&  \int_{\R^3 \times \Om} (1+p^2)dp dx \,  \eta_n \cdot \frac{f}{1+p^2}. \label{eqn:eta-Lp}
\end{align}
Taking $f=1$ and noting that $(1+p^2)^{-1} \in L^s((1+p^2)dpdx)$ for $s > 5/2$, we see that
\begin{align}
    \int_{\R^3 \times \Om} \eta = \lim_{n \rightarrow \infty} \int_{\R^3 \times \Om} \eta_n = \int_\Om \kappa.
\end{align}
By the same reasoning with $f\in L^\infty$, it follows that the weak convergence can be achieved on $L^1(\R^3\times \Om, dpdx)$.

Next, to see that $0 \leq \eta \leq 1$, let $\chi \geq 0$ denote any compactly supported bounded function with $L^1(\R^{3} \times \Om, dpdx)$ norm 1. Then
\begin{align}
    \int_{\R^3 \times \Om}  \eta \chi = \lim_{n \rightarrow \infty} \int \eta_n \chi.
\end{align}
Since $0 \leq \eta_n \leq 1$, we see that
\begin{align}
    0 \leq \int_{\R^3 \times \Om}  \eta \chi \leq 1.
\end{align}
It follows that $0 \leq \eta \leq 1$. 

It remains to show that 
\begin{align}
    \|\bar P L_1 \eta\|_{\dot{H}^{-1}(\Om)} < \infty.
\end{align}
Let $f \in \dot{H}^1(\Om)$ with mean zero. Then
\begin{align}
    \int_{\R^3 \times \Om} f \eta \leq& \|f\|_{L^6(\Om)} \int_{\R^3} (1+p^2)^{-5/6} (1+p^2)^{5/6}\|\eta^{6/5}\|_{L^1(\Om, dx)}^{5/6}.
\end{align}
By Hardy-Littlewood-Sobolev or Sobolev-Poincare and since $0 \leq \eta \leq 1$, we see that  
\begin{align}
    \int_{\R^3 \times \Om} f \eta \leq& \|f\|_{\dot{H}^1(\Om)} \|(1+p^2)^{-5/6}\|_{L^6(\R^3, dp)} \|(1+p^2)^{5/6} \|\eta^{6/5} \|_{L^1(\Om, dx)}^{5/6} \|_{L^{6/5}(\R^3,dp)} \\
    \leq&  C \|f\|_{\dot{H}^1(\Om)} \|(1+p^2)\eta^{6/5}\|_{L^1(dpdx)}^{5/6}.
\end{align}
Since $0 \leq \eta \leq 1$, we see that
\begin{align}
    \int_{\R^3 \times \Om} f \eta \leq& C \|f\|_{\dot{H}^1(\Om)} \|(1+p^2)\eta\|_{L^1(dpdx)}^{5/6}.
\end{align}
It follows by the Riesz representation theorem that
\begin{align}
    \|\bar P L_1 \eta\|_{\dot{H}^{-1}} \leq C \|(1+p^2)\eta\|_{L^1(dpdx)}^{5/6} < \infty.
\end{align}

Finally, Hardy-Littlewood-Sobolev or Sobolev-Poincare shows any $f \in \dot{H}^1(\Om)$ with mean zero is also in $L^6(\Om)$. For such an $f$,
\begin{align}
    \frac{f}{(1+p^2)} \in L^6((1+p^2)dpdx).
\end{align}
Weak convergence of $\eta_n$ to $\eta$ in $L^{6/5}((1+p^2)dpdx)$ shows weak convergence of $\bar PL_1\eta_n$ to $\bar PL_1\eta$ in $\dot{H}^{-1}(\Om)$.

In summary, $\eta_n$ converges weakly to $\eta$ in $L^p$ for $1 \leq p < \infty$ and $\bar PL_1\eta_n$ converges weakly in $\dot{H}^{-1}(\Om)$ to $\bar PL_1\eta$.

{\bf Step 4: lower semi-continuity.} 
Using \eqref{eqn:E-compact-form}, we write
\begin{align}
    \mathcal{F}_{\rm LSC}(\eta) =& L_1L_2 (h\eta) - \beta^{-1} L_1L_2 s(\eta).  \\
    & + \frac{1}{2} \left( \| \bar P L_1\eta \|^2_{\dot{H}^{-1}(\Om)} - 2 \text{Re}\lan L_1 \eta, (-\Delta)^{-1}\bar P\kappa \ran_{L^2(\Om)} + \| \bar P \kappa \|^2_{\dot{H}^{-1}(\Om)} \right) .
\end{align}
From this expression, it is evident that $\mathcal{F}_{\rm LSC}$ is convex in $\eta$. Hence, we may assume that $\eta_n$ converges strongly in $L^p$ for all $1 < p < \infty$. In particular, we may assume point-wise convergence. By Fatou's Lemma, the entropy term and all linear terms are lower semi-continuous. The Coulomb term $\| \bar P L_1\eta \|^2_{\dot{H}^{-1}(\Om)}$ is lower semi-continuous since it is the composition of a norm and integral operators.

{\bf Step 5: conclusion.}
By Steps 2 -- 3, we see that there is a minimizing sequence $\eta_n$ converging weakly to $\eta \in D_\kappa$. By Step 4, weak lower semi-continuity and convexity of $\mathcal{F}_{\rm LSC}$ shows that 
\begin{align}
    \mathcal{F}_{\rm LSC}(\eta) \leq \liminf_{n \rightarrow} \mathcal{F}_{\rm LSC}(\eta_n).
\end{align}
Since $\eta_n$ is a minimizing sequence, $\eta$ is a the minimizer (uniqueness is due to convexity). Step 1 shows that the associated $(\phi, \mu)$ of this minimizer solves \eqref{eqn:general-F-def} with \eqref{eqn:F-LSC-def}. The proof of Theorem \ref{thm:PL-existence-uniqueness} is now complete.
\end{proof}

\Addresses


\begin{thebibliography}{}


\bibitem{AACan} A. Anantharaman, E. Canc\`es (2009). Existence of minimizers for Kohn--Sham models in quantum chemistry, Ann. I. H. Poincar\'e -- AN {\bf 26 } 2425--2455.


\bibitem{ADFJM1} D. Arnold, G. David, M. Filoche, D. Jerison, and S. Mayboroda (2019). Computing spectra without solving eigenvalue problems. SIAM J. Sci. Comput., {\bf 41}(1), B69–B92. 



\bibitem{ADFJM2} D. Arnold, G. David, M. Filoche, D. Jerison, S. Mayboroda (2019). Localization of eigenfunctions via an effective potential. Comm. PDE.

\bibitem{ADFJM3} D. Arnold, G. David, D. Jerison, S. Mayboroda, M. Filoche (2016). Effective confining
potential of quantum states in disordered media, Phys. Rev. Lett., {\bf 116}, 056602.


\bibitem{BDDRV} Z. Bai, J. Demmel, J. Dongarra, A. Ruhe, and H. van der Vorst (Eds.) (2000). Templates for the solution of algebraic eigenvalue problems: a practical guide. Society for Industrial and Applied Mathematics.




\bibitem{CDL}
\'E. Canc\`es, A. Deleurence, and M. Lewin (2008). A new approach to the modeling of
local defects in crystals: The reduced Hartree-Fock case. Comm. Math. Phys.  {\bf 281}, 129–177.

\bibitem{CLL}
\'E. Canc\`es, S. Lahbabi, and M. Lewin (2013).  Mean-field models for disordered crystals.  J. Math. Pures Appl. {\bf 100(2)}, 241-274.

\bibitem{CSL}
\'E. Canc\`es, G. Stoltz, and M. Lewin (2006). The electronic ground-state energy problem: A new reduced density matrix approach. J. Chem. Phys. {\bf 125(6)}, 064101.


\bibitem{CLeBL}
I. Catto, C. Le Bris, P.-L. Lions (2001). On the thermodynamic limit for Hartree-Fock type models. Annales de l'Institut Henri Poincare/Analyse non lineaire, {\bf 18 }, Issue 6, 687 - 760.

\bibitem{CLeBL2} 
I. Catto, C. Le Bris, P. -L. Lions (2002). On some periodic hartree type models, Annales de l'Institut Henri Poincare/Analyse non lineaire, {\bf 19}, Issue 2, 143 - 190. 


\bibitem{CS} I. Chenn, I. M. Sigal (2020). On Derivation of the Poisson–Boltzmann Equation. J Stat Phys. J Stat Phys {\bf 180}, 954–1001.

\bibitem{CS1} I. Chenn, I. M. Sigal (2019). On Effective PDEs of Quantum Physics. In: D'Abbicco M., Ebert M., Georgiev V., Ozawa T. (eds) New Tools for Nonlinear PDEs and Application. Trends in Mathematics. Birkhäuser, Cham. 


\bibitem{CZ} I. Chenn, S. Zhang (2021). On the Reduce Hartree-Fock Equation with Anderson Type Background Charge Distribution. \texttt{[arXiv:2105.00295]}




\bibitem{DFM} G. David, M. Filoche, and S. Mayboroda (2019). The Landscape law for the integrated density of states. Submitted. \texttt{[arXiv:1909.10558]}



\bibitem{ELu} W. E, and J. Lu (2013). The Kohn-Sham equation for deformed crystals. Mem. Amer. Math. Soc. {\bf 221}, 1040. 



\bibitem{FM} M. Filoche, S. Mayboroda (2012). Universal mechanism for Anderson and weak localization. Proc. Natl. Acad. Sci. USA {\bf 109}, no. 37, 14761–14766.



\bibitem{PS1} M. Filoche, M. Piccardo, Y.-R. Wu, C.-K. Li, C. Weisbuch, S. Mayboroda (2017). Localization Landscape theory of disorder in semiconductors I: Theory and modeling. Phys. Rev. B {\bf 95}, 144204. doi:10.1103/PhysRevB.95.144204.

\bibitem{FNV}
R. L. Frank, P. T. Nam, and H. Van Den Bosch (2018). The ionization conjecture in Thomas–Fermi–Dirac–von Weizs\"acker theory. Communications on Pure and Applied Mathematics, {\bf 71(3)}, 577-614.




\bibitem{HK} P. Hohenber, W. Kohn (1964). Inhomogeneous electron gas. Phys. Rev. {\bf 136}: B864 - B871.



\bibitem{KS} W. Kohn, L. J. Sham (1965). Self-consistent equations including exchange and correlation effects. Physical Review. {\bf 140} (4A): A1133-A1138.

\bibitem{HS}  W. Kohn, C. D. Sherrill (2014). Editorial: reflections on fifty years of density functional theory, J. Chem. Phys. {\bf 140} 18A201.




\bibitem{Lang} S. Lang (1993). Real and Functional Analysis, Third Edition. Springer Graduate Texts in Mathematics, Series Volume
142, Springer-Verlag New York. 978-0-387-94001-4.


\bibitem{Lev} A. Levitt (2018). Screening in the finite-temperature reduced Hartree-Fock model. Arch. Ration. Mech. Anal.  {\bf 238.2}, 901-927.


\bibitem{Levy} M. Levy (1979). Universal variational functionals of electron densities, first order density matrices, and natural spin-orbitals and solutions of the $v$-representability problem. Proc. Natl. Acad. Sci. USA {\bf 76}: 6062 - 6065.


\bibitem{Levy2} M. Levy (1982). Electron densities in search of Hamiltonians. Phys. Rev. A {\bf 26}: 1200 - 1208.

\bibitem{LS}
M. Lewin, and J. Sabin (2015). The Hartree equation for infinitely many particles I. Well-posedness theory. Comm. Math. Phys. {\bf 334(1)}, 117-170.

\bibitem{Lieb3} E. H. Lieb  (1981). Thomas-Fermi and related theories of atoms and molecules, Rev. Modern Phys. {\bf 53}, no. 4, 603–641.

\bibitem{Lieb} E. H. Lieb (1983). Density Functionals for Coulomb Systems. International Journal of Quantum Chemistry, {\bf  24} , Issue 3 , 243 - 277.

\bibitem{LS1}
E.H. Lieb, B. Simon (1977). The Hartree-Fock theory for Coulomb systems, Commun. Math. Phys. {\bf 53}, 185–194.


\bibitem{LS2}
E.H. Lieb, B. Simon (1977). The Thomas-Fermi theory of atoms, molecules and
solids, Advances in Math. {\bf 23},  22-116.



\bibitem{MJ}
R. Matos, J. Schenker (2021).   Localization and IDS Regularity in the Disordered Hubbard Model within Hartree-Fock Theory,  Commun. Math. Phys. {\bf 382.3}, 1725-1768.




\bibitem{Nier} F. Nier (1993). A variational formulation of Schr\"{o}dinger-Poisson systems in dimension $d \leq 3$. Communications in Partial Differential Equations, {\bf 18} (7 and 8): 1125-1147. 




\bibitem{PS2} M. Piccardo, C.-K. Li, Y.-R. Wu, J. S. Speck, B. Bonef, R. M. Farrell, M. Filoche, L. Martinelli, J. Peretti, and C. Weisbuch (2017). Localization Landscape theory of disorder in semiconductors II: Urbach tails of disordered quantum well layers. Phys. Rev. B, 95, 144205.

\bibitem{PN} E. Prodan and P. Nordlander (2003). On the Kohn-Sham equations with periodic back-
ground potential. Journal of Statistical Physics {\bf 111}, issue 3-4, 967-992.




\bibitem{Saad} Y. Saad (1992). Numerical methods for large eigenvalue problems. Manchester University Press.


\bibitem{So1}
J.P. Solovej (1991). Proof of the ionization conjecture in a reduced Hartree-Fock
model, Invent. Math. {\bf 104},  291–311.

\bibitem{So2}
J.P. Solovej (2005).  The ionization conjecture in Hartree-Fock theory.  Annals of mathematics,  509-576.


\bibitem{WZ} W. Wang and S. Zhang (2020). The exponential decay of eigenfunctions for tight binding Hamiltonians via Landscape and dual Landscape functions. Annales Henri Poincar\'e. {\bf 22}, pages 1429–1457.



\end{thebibliography}
\end{document}